\documentclass[reqno, final,a4paper]{amsart}
\usepackage{color}
\usepackage[colorlinks=true,allcolors=blue
]{hyperref}
\usepackage{amsmath, amssymb, amsthm}
\usepackage{mathrsfs}
\usepackage{mathtools}
\usepackage[noabbrev,capitalize,nameinlink]{cleveref}
\crefname{equation}{}{}
\usepackage{fullpage}
\usepackage[noadjust]{cite}
\usepackage{graphics}
\usepackage{pifont}
\usepackage{tikz}
\usepackage{bbm}
\usepackage[T1]{fontenc}
\usetikzlibrary{arrows.meta}

\usepackage{environ}
\usepackage{framed}
\usepackage{url}
\usepackage[linesnumbered,ruled,vlined]{algorithm2e}
\usepackage[noend]{algpseudocode}
\usepackage[labelfont=bf]{caption}
\usepackage{framed}
\usepackage[framemethod=tikz]{mdframed}
\usepackage{pgfplots}
\usepackage{appendix}
\usepackage{graphicx}
\usepackage[textsize=tiny]{todonotes}
\usepackage{tcolorbox}
\usepackage[shortlabels]{enumitem}
\crefformat{enumi}{#2#1#3}
\crefrangeformat{enumi}{#3#1#4 to~#5#2#6}
\crefmultiformat{enumi}{#2#1#3}
{ and~#2#1#3}{, #2#1#3}{ and~#2#1#3}
\allowdisplaybreaks

\let\originalleft\left
\let\originalright\right
\renewcommand{\left}{\mathopen{}\mathclose\bgroup\originalleft}
\renewcommand{\right}{\aftergroup\egroup\originalright}

\newcommand\bfrac[2]{\left(\frac{#1}{#2}\right)}

\crefname{algocf}{Algorithm}{Algorithms}

\crefname{equation}{}{} %remove ``Equation''
\AtBeginEnvironment{appendices}{\crefalias{section}{appendix}} %appendices

\usepackage[color]{showkeys} %add in 'final' into parameter to remove showkeys

\colorlet{refkey}{orange!20}
\colorlet{labelkey}{blue!30}

\crefname{algocf}{Algorithm}{Algorithms}

\numberwithin{equation}{section}
\newtheorem{theorem}{Theorem}[section]

\newtheorem{lemma}[theorem]{Lemma}

\crefname{claim}{Claim}{Claims}

\newtheorem{conjecture}[theorem]{Conjecture}
\newtheorem*{question*}{Question}
\newtheorem{fact}[theorem]{Fact}

\theoremstyle{definition}
\newtheorem{definition}[theorem]{Definition}

\newtheorem{question}[theorem]{Question}
\newtheorem*{definition*}{Definition}

\theoremstyle{remark}
\newtheorem{remark}[theorem]{Remark}

\newcommand{\floor}[1]{\left\lfloor #1 \right\rfloor}
\newcommand{\ceil}[1]{\left\lceil #1 \right\rceil}

\newcommand{\one}{\mathbbm{1}}

\newcommand{\mb}{\mathbb}

\newcommand{\mbm}{\mathbbm}
\newcommand{\mc}{\mathcal}

\newcommand{\mr}{\mathrm}

\newcommand{\on}{\operatorname}

\renewcommand{\Pr}{\mb P}

\allowdisplaybreaks

\title{Partitioning problems via random processes}

\author[Anastos]{Michael Anastos}

\address{Institute of Science and Technology Austria (ISTA). Am Campus 1, 3400 Klosterneuburg, Austria.}
\email{michael.anastos@ist.ac.at}

\author[Cooley]{Oliver Cooley}

\address{Ludwig-Maximilians-Universit\"at M\"unchen, Mathematisches Institut. Theresienstrasse 39, D-80333 München, Germany.}
\email{cooley@math.lmu.de}

\author[Kang]{Mihyun Kang}

\address{Graz University of Technology, Institute of Discrete Mathematics. Steyrergasse 30, 8010 Graz, Austria.
}
\email{kang@math.tugraz.at}

\author[Kwan]{Matthew Kwan}

\address{Institute of Science and Technology Austria (ISTA). Am Campus 1, 3400 Klosterneuburg, Austria.}
\email{matthew.kwan@ist.ac.at}

\thanks{Michael Anastos was supported by the European Union’s Horizon 2020 research and innovation
programme under the Marie Sk\l{}odowska-Curie grant agreement No.\ 101034413.
Matthew Kwan was supported by ERC Starting Grant ``RANDSTRUCT'' No.\ 101076777, also funded by the European Union
\includegraphics[width=4.5mm, height=3mm]{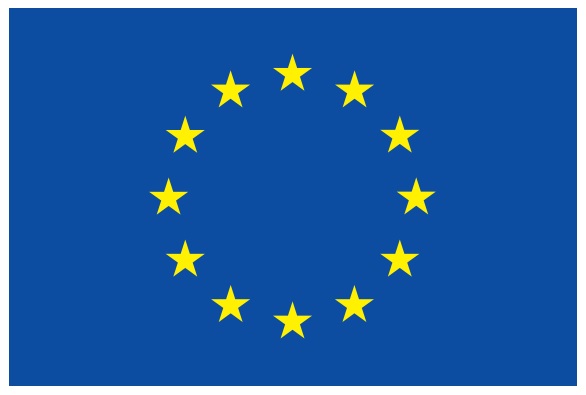}.
}

\begin{document}

\maketitle
\begin{abstract}
There are a number of well-known problems and conjectures about partitioning graphs to satisfy local constraints. For example, the \emph{majority colouring conjecture} of Kreutzer, Oum, Seymour, van der Zypen and Wood states that every directed graph has a 3-colouring such that for every vertex $v$, at most half of the out-neighbours of $v$ have the same colour as $v$. As another example, the \emph{internal partition conjecture}, due to DeVos and to Ban and Linial, states that for every $d$, all but finitely many $d$-regular graphs have a partition into two nonempty parts such that for every vertex $v$, at least half of the neighbours of $v$ lie in the same part as $v$. 

We prove several results in this spirit: in particular, two of our results are that the majority colouring conjecture holds for Erd\H os--R\'enyi random directed graphs (of any density), and that the internal partition conjecture holds if we permit a tiny number of ``exceptional vertices''.

Our proofs involve a variety of techniques, including several different methods to analyse random recolouring processes. One highlight is a \emph{personality-changing} scheme: we ``forget'' certain information based on the state of a Markov chain, giving us more independence to work with.

\medskip
\noindent{\textbf{Mathematics Subject Classification: }05C80, 05C15, 60C05}
\end{abstract}

\section{Introduction}

It is a classical fact (perhaps first proved by Lov\'asz; see \cite[pp.~237--238]{Lov66}) that every (finite\footnote{This fact is actually false for uncountably infinite graphs, and it
is a well-known open question whether it is true for countably infinite
graphs (this is the \emph{Unfriendly Partition Conjecture}; see \cite{SM90,AMP90}).}) graph has a red-blue colouring of its vertices, such that for every
red vertex, at least half of its neighbours are blue, and for each
blue vertex, at least half of its neighbours are red. Such red-blue
colourings are often called \emph{external partitions}, \emph{unfriendly
partitions} or \emph{disassortative} \emph{partitions}, and can be
interpreted from several different points of view. For example:
\begin{itemize}
\item A \emph{cut} of a graph is a partition of its vertices into two parts,
and the \emph{size} of a cut is the number of edges between the two
parts. Finding the maximum possible size of a cut is called the \emph{MAX-CUT
problem} and is of fundamental importance in computer science and
optimisation. External partitions correspond precisely to those cuts
which are \emph{locally} maximal, and have been studied extensively
in this context (see for example \cite{CGVYZ20,ABPW17,ET55,SY91,Pol95}).
\item The \emph{Ising model} is one of the central objects of study in statistical
physics. Given a graph $G$, with a real-valued \emph{interaction}
on each edge, the Ising
model describes a probability distribution over the set of configurations
of $\pm1$-valued \emph{spins} on the vertices of $G$, in terms of a \emph{Hamiltonian
}describing the \emph{energy} of each configuration. In this setting, an important question is to understand the \emph{locally energy-minimising} configurations (see for example \cite{ADLO19,GNS18,SGNS20,BM81,DGZ}). If the interactions
all take the same negative value (this is the \emph{antiferromagnetic}
regime), the locally energy-minimising configurations correspond
precisely to external partitions of $G$.
\item \emph{Minority dynamics} (also studied as the dynamics of the \emph{El Farol bar problem}; see for example \cite{CPG99,CDF22,CCCM09}) is a dynamical system on a social network where each person is coded either red or blue. In each round, each person changes their colour to the least popular colour among their neighbours. External partitions are precisely those colourings which are stable for minority dynamics.
\end{itemize}
Although it is a near-triviality to show that every graph has an external
partition, it is easy to obtain highly nontrivial questions by making
small changes to the definition of ``external partition''. The purpose
of this paper is to demonstrate how to make progress on various problems
of this type using probabilistic methods and ideas, especially \emph{random
recolouring processes}. Before we discuss our results, we start with
some background on some of the concepts and questions in this area.

\subsection{\label{subsec:internal}Internal partitions}

First, it is natural to consider the ``opposite'' of an external
partition: in an \emph{internal} partition, at least half of the neighbours
of every red vertex are red, and at least half of the neighbours of
every blue vertex are blue. These colourings correspond to locally
\emph{minimal} cuts, locally energy-minimising configurations in
the \emph{ferromagnetic} Ising model, or stable configurations of \emph{majority dynamics} (which has been much more thoroughly studied than minority dynamics; see the survey in \cite{MT17}). See the introduction of \cite{BL16} for a review of the graph theory literature on internal and external partitions.

Some graphs (such as stars or cliques) only have ``trivial'' internal
partitions (in which all vertices are the same colour). However, answering
a conjecture of Thomassen, it was proved by Stiebitz~\cite{Sti96} that
one can always find a nontrivial ``near-internal'' partition, where
each vertex has at most one more neighbour in its opposite colour
than its own.
A well-known conjecture in this area is that regular graphs without internal partitions are
extremely rare.
\begin{conjecture}
\label{conj:ban-linial}For every $d\in\mb{N}$, there are only finitely
many $d$-regular graphs with no nontrivial internal partition.
\end{conjecture}

As far as we can tell, this conjecture first appeared in print in a paper of Ban and Linial~\cite{BL16}, though it was previously posed in an open problem collection by DeVos~\cite{DeV}.
%(Ban and Linial attribute this conjecture to ``several investigators'' and DeVos ``suspects this problem has been considered previously'').

\cref{conj:ban-linial} is only known to hold for $d\in\{1,2,3,4,6\}$ (see \cite{SD02,BL16}). However, a weaker result is known to hold for all even $d$: adapting Stiebitz'
ideas, it was proved by Linial and Louis~\cite{LL20} that if $d$ is even,
then a vanishingly small proportion of $n$-vertex $d$-regular graphs
fail to have an internal partition (in the language of random graphs:
a random $d$-regular graph $\mb{G}_{\mathrm{reg}}(n,d)$ has a nontrivial
internal partition whp\footnote{We say a property holds \emph{with high probability}, or ``whp''
for short, if it holds with probability tending to 1. Here and for
the rest of the paper, all asymptotics are as $n\to\infty$, unless
stated otherwise.}).

\subsection{\label{subsec:bisections}Bisections}

A \emph{bisection} is a red-blue colouring where the numbers of red
and blue vertices are equal (or differ by one, if the total number
of vertices is odd). Although every graph has an external partition,
some graphs (such as stars) do not have external bisections. However,
one can ``come close'' for almost all graphs: resolving an old conjecture
due to F\"uredi, it was recently proved by Ferber, Kwan, Narayanan,
Sah and Sawhney~\cite{FKNSS} that almost all graphs have bisections in
which almost every vertex is externally coloured (and bisections in which almost every vertex is internally coloured).
To be precise: whp, an Erd\H os--R\'enyi graph\footnote{In the Erd\H os--R\'enyi random graph $\mathbb{G}(n,p)$ (perhaps more appropriately called the \emph{binomial} random graph), we fix a set of $n$ vertices and include each of the $\binom n2$ possible edges with probability $p$ independently.} $G\sim\mathbb{G}(n,1/2)$
has a bisection in which all but $o(n)$ vertices have at least half
of their neighbours in the opposite colour (and a bisection in which
all but $o(n)$ vertices have at least half of their neighbours in
their own colour).

This work has since been generalised in multiple directions. First,
via a beautiful application of Lindeberg's replacement trick (inspired
by work of Dembo, Montanari and Sen~\cite{DMS17}), and a delicate second
moment calculation (building on earlier results by Gamarnik and Li~\cite{GL18}), Dandi, Gamarnik and Zdeborov\'a~\cite{DGZ}
proved
that the conclusion of F\"uredi's conjecture holds even for
quite sparse random graphs: it holds whp for $\mb{G}(n,p_{n})$
as long as $np_{n}\to\infty$. In exciting recent work, Minzer,
Sah and Sawhney~\cite{MSS} finally managed to handle true internal/external
bisections (without exceptional vertices): with a sophisticated second-moment
calculation together with some ideas from the analysis of Boolean functions, they
proved that whp
$G\sim\mathbb{G}(n,1/2)$ has an external bisection and an internal
bisection. It is an open question whether this is also possible for
sparser (or denser) random graphs (though we note that graphs which are very close to being complete do not have internal bisections).
\begin{conjecture}
\label{conj:random}For any $p_{n}\in[0,1]$ (allowed to depend on
$n$), whp $G\sim\mathbb{G}(n,p_{n})$ has an external bisection. If $(1-p_n)n -\log n\to \infty$ then whp $G\sim\mathbb{G}(n,p_{n})$ also has an internal bisection.
\end{conjecture}

It would also be very interesting to find the above internal/external
bisections \emph{efficiently}\footnote{Many graph partitioning theorems (e.g., those in \cite{Sti96,FKNSS,LL20}) can be easily turned into efficient algorithms, but the proofs in \cite{MSS,DGZ} cannot, on account of their use of the second moment method.} (i.e., via a polynomial-time randomised algorithm).
Behrens, Arpino, Kivva and Zdeborov\'a~\cite{BAKZ22} used ideas from statistical
physics to study computational obstructions for certain types
of partitioning problems in random graphs; their work indicates that
the problem of finding an internal or external bisection does \emph{not}
have such an obstruction.

We remark that there are a number of other fascinating conjectures
about external and internal bisections that are less closely related
to our results in this paper: perhaps most notably, Bollob\'as and
Scott~\cite[Conjecture~8]{BS02} conjectured an analogue of Stiebitz' theorem for internal
bisections, and Ban and Linial~\cite[Conjecture~1]{BL16} conjectured that every
bridgeless cubic graph except the Petersen graph has an external bisection.

\subsection{\label{subsec:digraphs}Directed graphs}

Generalising the notion of an external partition to \emph{directed
graphs} (``digraphs''), we could ask for a red-blue colouring with the property that
for each vertex $v$, at most half of the \emph{out-neighbours} of
$v$ have the same colour as $v$. 
For general digraphs, it is not always
possible to find a colouring satisfying this property (e.g., odd directed cycles are counterexamples).
In a similar spirit to the last two subsections, one could try to
show that such a colouring is always ``almost'' possible, or that
such a colouring exists for almost all digraphs, but so far most of
the attention in this area has focused on adding \emph{additional
colours}.

Specifically, for any vertex-colouring of a digraph, we say that a
vertex is \emph{majority-coloured} if at most half of its out-neighbours
have the same colour as it. If every vertex is majority-coloured,
we say the colouring is a \emph{majority colouring}. The following
fascinating conjecture (typically known as the \emph{majority colouring
conjecture}) was made by Kreutzer, Oum, Seymour, van der Zypen and
Wood~\cite{KOSZW17}.
\begin{conjecture}
\label{conj:majority-colouring}Every directed graph has a majority
3-colouring.
\end{conjecture}

Progress on \cref{conj:majority-colouring} has come from a few different directions. First, Kreutzer,
Oum, Seymour, van der Zypen and Wood gave a simple proof that a majority
4-colouring always exists. They also observed that in various settings
\emph{random} 3-colourings can be useful, because in a random 3-colouring,
on average each vertex is the same colour as only a third of its neighbours.
Specifically, if an $n$-vertex digraph has minimum degree at least
about $\log n$, then a random 3-colouring is overwhelmingly likely
to be a majority colouring, and in a $d$-regular graph with $d\ge144$,
one can use the Lov\'asz Local Lemma to show that a random 3-colouring
has positive probability of being a majority 3-colouring. Also, via
consideration of a \emph{list-colouring} version of the problem, Anastos,
Lamaison, Steiner and Szab\'o~\cite{ALSS21} managed to prove the majority
colouring conjecture for directed graphs which have chromatic number\footnote{Here the chromatic number of a directed graph is simply the chromatic number of the graph obtained by removing the directions on the edges.}
at most 6.

Certain variations on the theme of majority colouring have also been
considered by various authors (see for example \cite{ABG17,GKP17,KS18,ABGGPZ20,ABGGPZ22,XSXCW22}).

\subsection{Results}

To state our results, it is convenient to introduce some notation. Given $\varepsilon\in (0,1)$ we say that a red-blue colouring of an $n$-vertex graph is an \emph{$\varepsilon$-almost-external
partition} (respectively, \emph{$\varepsilon$-almost-internal partition})
if all but $\varepsilon n$ vertices have at least half of their neighbours
in the opposite colour (respectively, the same colour). Say that a
vertex-colouring of an $n$-vertex directed graph is an \emph{$\varepsilon$-almost-majority
colouring} if all but $\varepsilon n$ vertices are majority-coloured.

Our first result is an approximate version of the internal partition conjecture
(\cref{conj:ban-linial}): we prove a weakening where a small number of exceptional vertices are allowed. On the other hand, we can demand that our partition is a bisection, and instead of considering $d$-regular graphs, we can consider the more general class of graphs whose maximum degree is at most $d$.
\begin{theorem}
\label{thm:ban-linial}Fix $d\in\mb{N}$ and $\varepsilon>0$. Then
there are only finitely many graphs with maximum degree at most $d$
which do not have an $\varepsilon$-almost-internal bisection. There
are also only finitely many such graphs which do not have an $\varepsilon$-almost-external
bisection.

In both cases, the desired bisections can be found via a randomised algorithm whose expected runtime is linear in the number of vertices.
\end{theorem}

Second, turning our attention to random graphs $\mb G(n,p_n)$, we prove that in the very sparse regime (where $p_n$ is of order at most $1/n$), whp there is an $o(1)$-almost-internal bisection and $o(1)$-almost-external bisection. This complements the result of Dandi, Gamarnik and Zdeborov\'a discussed in \cref{subsec:bisections}, which handles the regime where $p_n$ is of larger order than $1/n$.
\begin{theorem}
\label{thm:random}For any $p_{n}\in[0,1]$ such that $\lim\sup np_n<\infty$, whp $G\sim\mathbb{G}(n,p_{n})$ has an $o(1)$-almost-internal
bisection and an $o(1)$-almost-external bisection.

In both cases, the desired bisections can be found via a randomised algorithm whose expected runtime is linear in $n$. 
\end{theorem}

Next, consider the binomial random \emph{directed}
graph $\mathbb{D}(n,p)$ which has $n$ vertices, and each of the $n(n-1)$
possible directed edges are present with probability $p$ independently.
We prove the majority colouring conjecture for binomial random directed
graphs of any density. (This is most interesting in the very sparse case; as we discussed
in \cref{subsec:digraphs}, \emph{any} dense digraph can be easily shown to have
a majority 3-colouring).

\begin{theorem}
\label{thm:majority-Gnp}For any $p_{n}\in[0,1]$ (allowed to
depend on $n$), whp $D\sim\mathbb{D}(n,p_{n})$ has a majority 3-colouring.

This colouring can be found via a randomised algorithm whose expected runtime is linear in $n$.
\end{theorem}

We will say more about our proof techniques in \cref{sec:outline}, but to give a very brief impression: first, \cref{thm:random,thm:ban-linial} are proved in a unified way, via analysis of a random recolouring process (essentially, a ``lazy'' version of majority or minority dynamics). We also consider a random recolouring process in our proof of \cref{thm:majority-Gnp} (to find an \emph{approximate} majority-3-colouring), but some additional twists are required (in particular, we use a \emph{personality-changing} scheme: we ``forget'' certain information based on the state of a Markov chain, giving us more independence to work with). After finding our approximate majority 3-colouring, we modify it to obtain a genuine majority 3-colouring via analysis of a ``subcritical list-assignment process'', and some list-colouring ideas of Anastos, Lamaison, Steiner and Szab\'o~\cite{ALSS21}. Our subcriticality analysis also involves some new ideas, including a notion of a ``virtual process'' which simulates a small portion of our process under consideration.

In the spirit of the conjectures in \cref{subsec:internal,subsec:bisections}, we also show that for almost
every digraph, \emph{two} colours are \emph{almost} enough for a majority
colouring. In fact, it suffices to consider \emph{bisections}, where the numbers of red and blue vertices are as equal as possible.
\begin{theorem}
\label{thm:majority-2-colouring}
Let $p_n$ be such that $np_n(1-p_n)\to\infty$. Then whp $D\sim\mathbb{D}(n,p_n)$ has an $o(1)$-almost-majority bisection.
\end{theorem}

Unlike \cref{thm:ban-linial,thm:random,thm:majority-Gnp}, we do \emph{not} have a constructive proof of \cref{thm:majority-2-colouring} (our proof uses the second moment method, proceeding along similar lines to the work of Dandi, Gamarnik and Zdeborov\'a~\cite{DGZ}). In fact,
we believe that it is computationally intractable to find almost-majority
bisections in random directed graphs: our proof of \cref{thm:majority-2-colouring} can be modified to show that $o(1)$-almost-majority bisections satisfy the so-called \emph{overlap gap property}, introduced by Gamarnik (see for example \cite{Gam21}) as a heuristic certificate for computational intractability (see \cref{remark:OGP}).

\begin{remark}
    Regarding all of our theorems about random (di-)graphs (\cref{thm:majority-Gnp,thm:majority-2-colouring,thm:random}): we remark that there are two slightly different models of random graphs that are often collectively referred to as ``Erd\H os--R\'enyi random graphs''. We could either fix some $p\in [0,1]$ and include each (directed) edge with probability $p$ independently, or we could fix some integer $m$ and choose a random (di-)graph with exactly $m$ edges\footnote{The former type of random graph was actually first considered by Gilbert~\cite{Gil59}, slightly earlier than the groundbreaking work of Erd\H os and R\'enyi~\cite{ER60}, which considered the second model.}. There are seldom any important differences between the models (e.g., $G\sim \mb G(n,p)$ usually has about $p\binom n 2$ edges, and is essentially the same as a uniformly random graph with exactly $\lfloor p\binom n 2\rfloor$ edges). All the results in this paper hold equally well for both models, with minor changes to the proofs.
\end{remark}

\subsection{Further directions}\label{subsec:further}We believe that \cref{thm:ban-linial} is an important step on the path to a full proof
of the internal partition conjecture. We were able to obtain an approximate internal partition by running a random recolouring process for a small number of steps; it seems plausible that such processes, if run for long enough, tend to converge on an exact internal partition. However, the longer we wish to run such a process, the harder it is to analyse its behaviour. Related issues are encountered in the study of majority dynamics (see for example \cite{BCOTT16}): there are a number of open problems concerning the long-term behaviour of majority dynamics in various settings.

It might be possible to sidestep the above issue, and to combine an approximate result of the type in \cref{thm:ban-linial} with a separate ``completion'' step (we were successful in doing this in our proof of \cref{thm:random}). In particular, it may be possible to prepare ``gadgets'' separately to our random recolouring process, that assist with transforming an approximate solution into an exact one (cf.\ the \emph{absorption} method; see \cite{Sze13}). Note that one can always obtain an internal partition by iteratively making local improvements; the challenge is to ensure that this internal partition is nontrivial. So, we might imagine some arrangement of gadgets that ``protects'' a subset of vertices, ensuring that it can never become all-red or all-blue. However, it is unclear how to actually implement this type of idea, without making very strong assumptions about the structure of our graph.

We also envision a path to the full majority colouring conjecture (for arbitrary digraphs) via random recolouring processes,
but for this our random recolouring analysis needs to be made much
more robust. \cref{thm:majority-Gnp} is stated only for random digraphs, but as discussed in \cref{rem:high-girth}, our methods could conceivably be
generalised to arbitrary digraphs with high \emph{girth}. Without
a girth assumption it seems one would need quite different methods
to analyse random recolouring processes on digraphs.

Also, \cref{thm:majority-2-colouring} suggests that there is a lot of ``room'' in the majority
colouring conjecture, and that two colours are very nearly enough.
For example, we see no obvious reason why a Stiebitz-type theorem for ``near-majority''
2-colourings (or even bisections) should not be possible, as follows.
\begin{question}
Is it true that \emph{every} digraph $D$ has a bisection (or at least a 2-colouring) such that
each vertex has at most one more out-neighbour in its own colour than
the opposite colour?
\end{question}

\section{Outline of the paper and proofs}\label{sec:outline}

In this section we sketch the ideas in the proofs of \cref{thm:majority-Gnp,thm:ban-linial,thm:random,thm:majority-2-colouring}. The proof of \cref{thm:majority-Gnp} is much more involved than the others,
largely because \cref{thm:majority-Gnp} is about \emph{exact} majority-colourings, while the other theorems are only concerned with \emph{approximate} internal/external/majority colourings. 

\subsection{Majority 3-colouring of random digraphs} As a starting point for our proof of \cref{thm:majority-Gnp}, note that if we consider a \emph{uniformly random} 3-colouring, then every vertex is majority-coloured with probability at least 2/3. (This probability gets closer and closer to 1 as the out-degrees get larger, i.e., as the arc-sampling probability $p$ in $\mb D(n,p)$ gets larger.)

In the regime where $p$ has order of magnitude $1/n$ (which is our main regime of interest), in a random 3-colouring we expect that a non-negligible fraction of vertices will fail to be majority-coloured. We can hope to improve the situation by \emph{randomly recolouring} those vertices to a different colour. One can do an explicit calculation to see that the expected proportion of majority-coloured vertices does increase after such a recolouring, though we still expect there to be some vertices which are not majority-coloured (for example, the recolouring could have caused a vertex which was originally majority-coloured to no longer be majority-coloured).

\subsubsection{Random processes}
It is then natural to consider a random greedy recolouring \emph{process}, which repeatedly checks which vertices fail to be majority-coloured, and randomly recolours them: one may hope that this process tends to converge to a majority 3-colouring.
In principle, one can explicitly compute the expected proportion of majority-coloured vertices after any finite number of steps, but the formulas get out of hand very rapidly (na\"ively, the complexity of the formulas grows doubly-exponentially in the number of steps, though since the process is Markovian, one can use ideas of Lacker, Ramanan and Wu~\cite{LRW} to reduce this to a single-exponential dependence).

It is however possible to prove a (somewhat crude) \emph{recursive upper bound} on the (asymptotic) expected proportion $f_t$ of non-majority-coloured vertices at time $t$ (where we view $t$ as being fixed while $n\to\infty$), as follows. First, note that our process is ``local'': for the colour of a vertex $w$ to influence the colour of a vertex $v$ in $t$ steps, there must be a directed path from $v$ to $w$ of length at most $t$. Sparse random graphs are known to have very few short cycles, so in a typical outcome of $D\sim \mb D(n,p)$, there are very few pairs of length-$t$ directed paths that intersect after starting at two different out-neighbours of the same vertex. We can use this to deduce that for most vertices $v$, the colours of the out-neighbours of $v$ (until time $t$) are \emph{independent}. Now, note that whenever a vertex $v$ is not majority-coloured after $t$ steps, it must have happened that some colour ``overtook'' as the majority colour at time $t$ (i.e., that colour appeared on at most half of the out-neighbours of $v$ at time $t-1$, then more than half at time $t$). Since a $f_{t-1}$-fraction of vertices change their colours between time $t-1$ and time $t$, and since almost all vertices have out-neighbours with independent colours, we can explicitly compute (in terms of $f_{t-1}$) the expected proportion of vertices witnessing an overtaking event, which gives us an upper bound for $f_t$.

\subsubsection{Personality-changing}If $p=\lambda/n$ with (say) $\lambda>20$, we can perform some careful analysis on fixed points of the above recurrence, to deduce that $\lim_{t\to \infty}f_t=0$. However, our recurrence is not strong enough to handle all $p$ (e.g., when $p=10/n$, we can prove that our recurrence converges to a nonzero fixed point). Clearly, our recurrence is wasteful (essentially, we are computing the probability that \emph{any} colour overtakes at a vertex $v$, as an upper bound on the probability that \emph{the colour of~$v$} overtakes at~$v$). However, due to a lack of independence it does not seem tractable to modify our recurrence to take this inefficiency into account.

Instead, we modify the process to intentionally ``forget'' pertinent information, in a way that seemingly makes it perform worse, but which introduces independence that makes it possible to prove a stronger recurrence. Specifically, we introduce an auxiliary Markov chain at each vertex which describes the ``personality'' of the vertex at a given point in time (the personality describes whether all available information is used to decide whether to change colour, or whether certain information is intentionally ignored). With these ideas, we are able to prove a stronger recurrence which allows us to prove $\lim_{t\to \infty}f_t=0$ without a lower bound on $p$ (here $f_t$ is the asymptotic expected proportion of non-majority-coloured vertices after time $t$, in our modified recolouring process).

\subsubsection{List colouring and subcriticality}
We are not yet done: with the above ideas, one can only find a 3-colouring of $D\sim \mb D(n,p)$ such that \emph{almost all} vertices are majority-coloured. Indeed, there will typically be a small number of vertices whose local neighbourhood has pathological structure not amenable to the above analysis, and there is a limit on the number of steps we can control before our recursive analysis breaks down (we can let $t$ grow with $n$, but not very rapidly).

It is well-known that small subsets of sparse random graphs tend to have very simple structure (in $D\sim \mb D(n,p)$, whp any set of $o(n)$ vertices has average out-degree at most $1+o(1)$), so we can hope to take advantage of this structure to ``manually fix'' the ``exceptional'' vertices which are not majority-coloured. In particular, it is not hard to show that the subgraph induced by the exceptional vertices has chromatic number at most 6; recall from \cref{subsec:digraphs} that (as proved by Anastos, Lamaison, Steiner and Szab\'o~\cite{ALSS21}), such digraphs have a majority 3-colouring. Of course, it does not suffice to find a majority 3-colouring of the exceptional vertices in isolation: we must make sure that the exceptional vertices are coloured in a way that is ``compatible'' with the previously coloured non-exceptional vertices. It turns out that we will indeed be able to use the ideas of \cite{ALSS21}, but significant additional work is required.

First, instead of directly using the main result of \cite{ALSS21}, we extract a more general statement from its proof: if we assign a \emph{pair} of colours (i.e., a list of size 2) to each of the vertices of a digraph, and if for each possible list there are no directed cycles among the vertices with that list, then we can find a majority colouring assigning each vertex a colour from its list.

In order to actually apply this result we need a lot more information about the colouring produced by our random process (specifically, among the nonexceptional vertices which are already majority-coloured, we need to understand whether they would become non-majority-coloured if certain choices were made for the colours of the exceptional vertices). To this end we define a ``list-assignment process'' which expands the set of exceptional vertices, assigning lists as it goes, and ensuring that the non-exceptional vertices are majority-coloured no matter what colour is assigned from the lists of the exceptional vertices.

In order to study this list-assignment process (and in particular, to show that only a small number of vertices are assigned lists), we need a number of different ideas. In particular, we introduce the notion of a ``virtual process'' which ``simulates'' a small part of our actual list-assignment process of interest. We are able to show (via comparison to subcritical branching process, and a union bound) that whp \emph{all possible} virtual list-assignment processes do not introduce too many exceptional vertices, and we are separately able to show that our list-assignment process can be ``covered'' by a small number of virtual processes (roughly speaking, we need to show that two different types of growth are bounded in terms of each other).

We remark that related ``self-bounding via subcriticality'' ideas appeared in previous work of Cooley, Lee and Ravelomanana~\cite{CLR}, studying \emph{warning propagation} on random graphs. Also, it is worth noting that the algorithmic proof of the \emph{Lov\'asz Local Lemma} due to Moser and Tardos~\cite{MT10} (which has already been applied to the majority colouring conjecture~\cite{KOSZW17}) proceeds by a related subcriticality analysis of a certain ``recolouring process''; our recolouring process can be viewed as being more efficient but much more difficult to analyse.

\subsection{Non-constructive majority 2-colouring of random digraphs}
The proof of \cref{thm:majority-2-colouring} proceeds along very similar lines as the proofs in \cite{DGZ,GL18}, but it turns out that the relevant computations are much easier in the setting of random digraphs than the setting of random graphs. Most of the effort goes towards estimating the  second moment of the number of majority bisections, which boils down to a large-deviations computation. This estimate is not strong enough to prove \cref{thm:majority-2-colouring} directly, but it can be ``boosted'' using a concentration trick due to Frieze~\cite{Fri90}.

\subsection{Internal and external bisections}
The proofs of \cref{thm:ban-linial,thm:random} are essentially the same as each other. For concreteness, we discuss the ``internal'' part of \cref{thm:ban-linial} (i.e., we describe how to find an $\varepsilon$-internal bisection in an $n$-vertex graph with maximum degree at most $d$).

The crucial observation is that internal partitions correspond precisely to cuts which are locally minimal: if we start with any cut which does not correspond to an internal partition, then it is possible to flip the colour of some vertex to decrease the size of the cut\footnote{Nothing analogous to this seems to be true for majority colourings of digraphs!}. If we repeatedly flip colours in this way, we will always end up with an internal partition; the challenge is to make sure that the colour classes are not too imbalanced (for example, if we are able to find an internal partition in which the sizes of the colour classes differ by at most $(2\varepsilon/d) n$, then we can flip at most $\varepsilon d/n$ vertices to obtain an $\varepsilon$-internal bisection).

One might try to carefully design an algorithm that chooses which vertices to flip, in which order, in such a way that the two colour classes stay balanced. However, as far as we can tell this seems to be completely intractable in general. Instead, we make choices \emph{randomly} (and choose the initial red-blue colouring randomly as well). The idea is that if there is no particular bias towards red or blue then we should end up with a cut which is roughly half-red and half-blue.

It seems plausible that if one repeatedly chooses a uniformly random flip among all flips which would decrease the size of the cut, then whp the resulting internal partition is nearly a bisection. However, it is far from obvious how to prove this: one must track the process for a rather long time (and there is nontrivial dependence between the steps), and there is no obvious ensemble of statistics that drive the process with which one might hope to use the \emph{differential equations method} (which is a standard way to study the trajectory of combinatorial random processes; see \cite{Wor99DE}). Instead, we modify the process slightly, flipping large \emph{batches} of vertices at once. As long as the batches are not too large, one can show that the flips typically do not interfere with each other very much, and the size of the cut decreases quite dramatically with each batch of flips. So, our process runs for only a very small number of steps, and as a result the dependencies are mild enough to apply a standard concentration inequality to the numbers of red and blue vertices at each step. Similar ideas (in the setting of dense random graphs, with a much more complicated implementation) were used in \cite{FKNSS}.

We remark that our tuning of the batch size (not too small that we lose control over concentration, and not so large that the flips interfere with each other) may be compared with tuning of the \emph{learning rate} in gradient descent and similar optimisation algorithms. We also remark that the general idea of splitting a random process into batches also features in the celebrated \emph{R\"odl nibble}~\cite{Rod85} in probabilistic combinatorics, though the purpose of the batches is rather different.

\subsection{Organisation}
The proof of \cref{thm:majority-Gnp} spans \cref{sec:majority-lemmas,sec:processes,sec:marking,sec:acyclic-partition,sec:finish}. Specifically, some key lemmas are stated in \cref{sec:majority-lemmas}, then our random recolouring process is described in \cref{sec:processes}, then the list-assignment process is described and studied in \cref{sec:marking}, a list-colouring theorem is proved in \cref{lem:acyclic-partition}, and everything is put together in \cref{sec:finish}. \cref{thm:majority-2-colouring} is proved in \cref{sec:majority-2-colouring}, and \cref{thm:random,thm:ban-linial} are proved in \cref{sec:internal-external}. We also have two appendices with the details of various routine calculations.

\section{Key lemmas for majority 3-colouring}\label{sec:majority-lemmas}
In this section we outline the ingredients in the proof of \cref{thm:majority-Gnp}. We restrict our attention to the case $p=O(1/n)$ (if $p\ge C/n$ for a suitably large constant $C$, we can prove \cref{thm:majority-Gnp} with a much cruder version of the arguments outlined in this section, as we will see in \cref{sec:finish}).

First, the following lemma tells us that whp we can majority-colour \emph{almost} all the vertices of a random graph, and is proved by an iterative recolouring process.

\begin{lemma}\label{lem:local-output}
Fix any constant $C>0$ and let $D\sim\mb D(n,p)$ for $p\le C/n$. Then, whp $D$ has a 3-colouring in which all but $o(n)$ vertices are majority-coloured.
\end{lemma}
We prove \cref{lem:local-output} in \cref{sec:processes}. In that section, we first discuss how to crudely study a na\"ive recolouring process via a recurrence, and then we show how to use a ``personality-changing'' Markov chain to strategically forget information, eliminating certain dependencies and facilitating a sharper analysis.

Actually, we remark that we include the statement of \cref{lem:local-output} purely for exposition. For the rest of the proof of \cref{thm:majority-Gnp}, we will not really need the statement of \cref{lem:local-output} \emph{per se}; rather, we will need the analysis of the random recolouring process in its proof.

Unfortunately, no matter how long we run our process we cannot rule out the possibility that some small number of vertices fail to be majority-coloured. However, we can benefit from the fact that small subsets of sparse random graphs have very simple structure, as follows.

\begin{lemma}\label{lem:2degenerate}
For any constants $\varepsilon,C>0$ there is $\delta>0$ such that the following holds. If $G\sim\mb G(n,p)$ for $p\le C/n$, then whp every vertex subset $S$ with $|S|\le \delta n$ spans at most $(1+\varepsilon)|S|$ edges.
\end{lemma}
\cref{lem:2degenerate} follows from a routine calculation (which appears for example in \cite[Proof of Theorem~1]{JT08}).

Recall that a graph is \emph{$k$-degenerate} if every subgraph has a vertex with degree at most $k$. Such graphs have chromatic number at most $k+1$. Note that \cref{lem:2degenerate} (applied with any $\varepsilon<1/2$) implies that every subgraph of $\mb G(n,p)$ with at most $\delta n$ vertices is $2$-degenerate, therefore has chromatic number at most $3$. Instead of chromatic number, we will need to use \cref{lem:2degenerate} to establish a somewhat more delicate partitioning property, as in the following lemma (adapted from work of Anastos, Lamaison, Steiner and Szab\'o~\cite{ALSS21}, and proved in \cref{sec:acyclic-partition}).

\begin{lemma}\label{lem:acyclic-partition} Let $D$ be a digraph with a distinguished vertex subset $U$, and assign to each $v\in U$ a list $L(v)\subseteq\{1,2,3\}$ of size 2. Suppose that for each of the three possible lists, there is no directed cycle among the vertices in $U$ with that list. Then for any 3-colouring $c:V(D)\setminus U\to \{1,2,3\}$ (of the vertices without lists), we can complete the colouring by assigning a colour $c(v)\in L(v)$ to each $v\in U$, in such a way that every $v\in U$ is majority-coloured.
\end{lemma}
We emphasise that in \cref{lem:acyclic-partition} we make no guarantees about the majority-colouredness of the vertices not in $U$, although when we apply this lemma some appropriate conditions will indeed be satisfied.

In order to apply \cref{lem:acyclic-partition} (with \cref{lem:2degenerate}), we need a lot more information about the colouring produced by \cref{lem:local-output}. To this end we define a ``list-assignment process'' which builds on the colouring from \cref{lem:local-output}, assigning lists to a small subset of vertices based on all the knock-on effects that would result from changing the colours of the initially non-majority-coloured vertices. (Crucially, we show that these knock-on effects can be compared to a \emph{subcritical} branching process). The outcome of our list-assignment process is as follows.

\begin{lemma}\label{lem:marking}
Fix any constant $C\ge 0.1$ and let $D\sim\mb D(n,p)$ for any $0.1/n\le p\le C/n$. Whp we can find a subset $U\subseteq V(D)$, an assignment of a colour $c(v)\in \{1,2,3\}$ to each $v\notin U$, and an assignment of a list $L(v)$ to each $v\in U$, such that the following hold.
\begin{enumerate}[{\bfseries{L\arabic{enumi}}}]
    \item\label{L1} For any completion of our partial colouring $c$, obtained by assigning a colour $c(v)\in L(v)$ to each $v\in U$, we have that every $v\notin U$ is majority-coloured with respect to $c$ (i.e., the initial partial colouring is a ``robust'' majority colouring, in the sense that the vertices outside $U$ remain majority-coloured no matter what we do inside $U$).
    \item\label{L2} $|U|=o(n)$.
    \item\label{L3} Each $v\in U$ has list size $|L(v)|=2$ or $|L(v)|=3$.
    \item\label{L4} Every directed cycle in $D[U]$ contains at least two vertices with list size 3.
    \item\label{L5} There is $\ell=O(1)$ such that every length-$\ell$ directed path in $D[U]$ contains a vertex with list size~3.
\end{enumerate}
\end{lemma}
\begin{remark}
    Note that in \cref{lem:marking} we assume the lower bound $p\ge 0.1/n$. This is not very crucial (it just makes technical considerations slightly more convenient in a minor part of the proof). Due to this assumption, in our proof of \cref{thm:majority-Gnp} we will treat the regime $p<0.1/n$ separately.
\end{remark}

We prove \cref{lem:marking} in \cref{sec:marking}. In \cref{sec:finish} we then show how to combine \cref{lem:marking,lem:acyclic-partition,lem:2degenerate} to prove \cref{thm:majority-Gnp}.

\section{Majority-colouring processes}\label{sec:processes}

In this section we prove \cref{lem:local-output}. First we
describe a simple random process that attempts to majority-colour
a graph (but which is intractable to analyse exactly), and then we
describe a modification which permits somewhat sharper analysis. In this section we sometimes refer to the three colours 1,2,3 as ``red'', ``green'' and ``blue''.

\subsection{A simple process}

The most obvious candidate to prove \cref{lem:local-output}
is the process that first randomly chooses an initial colour for each vertex, and then repeatedly changes the colour of every vertex that is not majority-coloured (making relevant choices randomly). That is to say, at each time step, we identify the set of all vertices which are not majority coloured, and we independently change each of them to some random other colour (simultaneously).

We were not able to use this simple process to prove \cref{lem:local-output}, but it is nonetheless instructive to see what bounds we can prove with it (as a warm-up for the next subsection, where we introduce a more sophisticated process). The first key observation is that this is a ``local'' process: in order to know the colour of a vertex $v$ after $t$ random recolouring steps, we only need to know about the colours of vertices which can be reached from $v$ by directed paths of length at most $t$. Crucially, $\mb D(n,p)$ locally converges to a $\on{Poisson}(np)$ Galton--Watson tree, as follows. We write $\on d_{\mathrm{TV}}(X,Y)$ for the total variation distance\footnote{The \emph{total variation distance} between two (discrete) probability distributions $\mu,\nu$, taking values in a space $\Omega$, is the supremum of $|\mu(A)-\nu(A)|$ over all $A\subseteq \Omega$.} between two random objects $X$ and $Y$.
\begin{lemma}\label{lem:GW}
Fix a constant $C>0$ and let $p\le C/n$. Let $T$ be a $\on{Poisson}(np)$ Galton--Watson
tree (with root $r$, say), and orient all the edges of $T$ away from the root $r$. Let $D\sim \mb D(n,p)$. For a random vertex $v\in V(D)$, let $D_t(v)$ (respectively $T_t$) be the subgraph of $D$ (respectively, of $T$) induced by those vertices reachable by directed paths of length at most $t$ from $v$ (respectively, from $r$). Then for constant $t$ we have $\on d_{\mathrm{TV}}(D_t(v),T_t)\le O((\log n)^{2t}/n)$.
\end{lemma}
\begin{proof}[Proof sketch]
We will be crude and brief with details, as very similar observations have been made many times in the literature (see for example \cite[Theorem~6]{Cur17}). By a Chernoff bound, with probability $1-o(1/n)$ each vertex in $D$ and $T_t$ has degree at most $\log n$. In particular, considering breadth-first search in $D$ starting from $v$, the probability that we reach any vertex via two different paths of length at most $t$ (i.e., the probability that $D_t(v)$ is not one of the possible outcomes of $T_t$) is at most $p(\log n)^{2t}=O((\log n)^{2t}/n)$.

Now, the out-degree of each vertex in $D$ has a $\on{Binomial}(n-1,p)$ distribution, while the out-degree of each vertex in $T$ has a $\on{Poisson}(np)$ distribution. By standard estimates (see for example \cite[Eq.~(1.1)]{BHJ92}) we have $\on d_{\mathrm{TV}}(\on{Binomial}(n-1,p),\on{Poisson}(np))=O(np^2)=O(1/n)$. So, among possible outcomes $R$ of $T_t$ which have maximum degree at most $\log n$, we have $\Pr[D_t(v)=R]\le \Pr[T_t=R]+O((\log n)^{t}/n)$ (noting that each such $R$ has at most $(\log n)^t$ vertices). There are $O((\log n)^t)$ different outcomes of $R$ to consider, so the desired result follows.
\end{proof}

Morally speaking, \cref{lem:GW} says that to understand the behaviour of our random recolouring process for $t=o(\log n/\log \log n)$ steps, it suffices to consider an analogous process on a $\on{Poisson}(np)$ Galton--Watson
tree.

Now, say ``time $t$'' is the moment in time just before the $t$-th recolouring step (so at time $1$, the colouring is uniformly random).
Say that a colour (say, red) \emph{overtakes} for a vertex $v$ at
time $t$ if the following holds: at time $t-1$, there are at most
$\deg^{+}(v)/2$ red vertices in the out-neighbourhood of $v$, but
at time $t$ there are more than $\deg^{+}(v)/2$ of them. Then,
a vertex $v$ fails to be majority-coloured at time $t$ if and only
if the following holds: there is some colour $\gamma$ such that $v$ has
colour $\gamma$ at time $t$, and $\gamma$ overtakes for $v$ at time $t$.

Given $\lambda\ge 0$ and $t\in \mathbb N$, let $f_{t}$ be the probability that if we perform our random recolouring process
on a $\on{Poisson}(\lambda)$ Galton--Watson tree $T$, the root $r$ is not majority-coloured
at time $t$ (meaning that it will be randomly recoloured at the $t$-th recolouring step). 
This probability is subject to the randomness of $T$, and also the randomness of the recolouring process.  We have 
\[ f_1=\Pr[r\text{ is not majority-coloured at time $t=1$}]\] and for $t>1$
\begin{align*}
f_{t} & =\sum_{\gamma}\Pr[r\text{ has colour }\gamma\text{ at time }t\text{, and }\gamma\text{ overtakes for }r\text{ at time }t]\\
 & \le3\,\Pr[\text{red overtakes for }r\text{ at time }t]\\
 & =3P_\lambda(f_{t-1}),
\end{align*}
where 
\[
P_\lambda(f)=\sum_{d=0}^{\infty}\frac{e^{-\lambda}\lambda^{d}}{d!}Q_d(f)
\]
and $Q_d(f)$ is defined to be
\begin{equation}\sum_{i=1}^{\floor{d/2}}\binom{d}{i}\left(\frac{1}{3}\right)^{i}\left(\frac{2}{3}\right)^{d-i}\sum_{j=\floor{d/2}+1}^{d}\sum_{k=0}^{i}\binom{i}{k}f^{k}(1-f)^{i-k}\binom{d-i}{j+k-i}\left(\frac{f}{2}\right)^{j-i+k}\left(1-\frac{f}{2}\right)^{d-j-k}.\label{eq:Qd}\end{equation}
To explain the formula for $Q_d(f)$, consider a single vertex $v$ with $d$ out-neighbours $v_1,\dots,v_d$. Suppose that the colours of $v_1,\dots,v_d$ are initially independently randomly chosen from $\{\text{red, green, blue}\}$, and then every vertex independently decides to change its colour (to a different one, chosen randomly) with probability $f$. Then, the probability that red overtakes as the majority colour among $v_1,\dots,v_d$ is precisely $Q_d(f)$. Indeed, in the formula for $Q_d(f)$ we represent by $i$ the possible
numbers of red vertices before the overtaking event, and
by $j$ the possible numbers of red vertices after the overtaking
event. We represent by $k$ the possible numbers of red
vertices being randomly recoloured during the overtaking
event (so $j-i+k$ non-red vertices must be recoloured to red).

To explain the rest of the above formulas: note that if we condition on the out-degree of $r$ being $d$, and we consider the $d$ disjoint subtrees $T_1,\dots,T_d$ rooted at the $d$ out-neighbours $v_1,\dots,v_d$ of $r$, then $T_1,\dots,T_d$ (together with their vertex-colourings at time $t-1$) are independent and have the same distribution as $T$. For each $i$, the colour of $v_i$ changes at time $t$ if and only if $v_i$ is not majority-coloured at time $t-1$, which happens with probability $f_{t-1}$ (independently for each $i$). So, the conditional probability that red overtakes for $r$ at time $t$ is precisely $Q_d(f_{t-1})$. The out-degree of $r$ is $\on{Poisson}(\lambda)$-distributed, so $P_\lambda(f_{t-1})$ is the unconditional probability that red overtakes for $r$ at time $t$.

Now, one can check that $P_\lambda$ is monotone increasing, so
$3P_\lambda(3P_\lambda(\dots 3P_\lambda(f_{1})\dots))$ is an upper bound for $f_{t}$.
Unfortunately, an explicit computation shows that (say) $P_{10}'(0)>1$, so when $\lambda=10$, this recurrence will not tend to zero as $t\to\infty$ (it seems that this unfortunate situation happens when $\lambda$ is a real number in the approximate range $3\le \lambda\le 20$).

\subsection{Personality-changing to improve the recurrence}\label{subsec:personality-changing}
Note that the above analysis features a lossy union bound (roughly speaking, this costs us a factor of 3 in the recurrence). In order to improve the above analysis, we would ideally like to prove
a nontrivial upper bound on
\begin{equation}
\Pr[r\text{ is red at time }t\;|\;\text{red overtakes for }r\text{ at time }t].\label{eq:conditional-probability}
\end{equation}
Note that it is not too hard to describe the evolution of the colour of our root vertex $r$, given the history of colours of overtaking-events of its out-neighbours. Indeed, every time $\gamma$ overtakes, we ask if $r$ has colour $\gamma$, and if so we change it to a random other colour. We might hope to obtain a nontrivial bound on the above probability by conditioning on an arbitrary possible history
of overtaking-colours (determined by the colours of the out-neighbours of $r$), and proving a uniform upper bound on the probability that
$r$ has colour $\gamma$ at time $t$ given this particular history. Unfortunately,
for some very pathological histories we cannot get a nontrivial bound
this way (e.g., given the history $\dots,1,2,1,2,1,2$ of overtaking-colours,
we can be almost certain that $r$ has colour $3$ at time $t$), and it
seems to be difficult to say anything nontrivial about the distribution
of the overtaking-colour-history.

Instead, we consider a variation of our recolouring process (which is still ``local'', but is no longer ``Markovian'': transition probabilities will now depend on the entire history of the colours of the out-neighbours).

At each point in time, every vertex now has a ``personality'' as well as a colour: it can be \emph{paranoid}
or \emph{thoughtful}. The idea is that paranoid vertices randomly change
their colour at each overtaking event, regardless of their own colour (i.e., they change even if they don't have to).
On the other hand, thoughtful vertices do take their own colour into account, and may not switch if they don't have to. However, if a vertex is thoughtful and decides not to randomly change its colour at an overtaking event, then we change its personality to paranoid.
(So, if a vertex $v$ is thoughtful, then we can guarantee that the colour of $v$ was determined
by a random resampling from the last overtaking event. This means its colour is uniformly random among the two colours different from the colour that overtook at the last overtaking event).

Specifically: initially, set all vertices to be paranoid with
probability $1/3$, and thoughtful with probability $2/3$ independently of each other. Then, when a colour $\gamma$ overtakes for a vertex
$v$ we proceed as follows.

\begin{itemize}
\item If $v$ is thoughtful, we first define a real number $p^*\in \{0,1/3,1/2\}$ (which can be interpreted as ``the probability that $v$ has colour $\gamma$, given all relevant information except the current colour of $v$'').
\begin{itemize}
\item If $v$ has never changed its colour before, then $p^*=1/3$.
\item If, before the last time $v$ changed colour, it had colour $\gamma$, then $p^*=0$.
\item Otherwise, $p^*=1/2$.
\end{itemize}
Now, flip a biased coin that comes up heads with probability $(1/2-p^*)/(1-p^*)$ (so, informally, ``given all relevant information except the current colour of $v$'', the probability that $v$ has colour $\gamma$ or that the coin came up heads is exactly $1/2$).
\begin{itemize}
\item If $v$ has colour $\gamma$ or the coin came up heads, then randomly recolour $v$ to a colour other than $\gamma$, but do not change its personality.
\item Otherwise, do not change the colour of $v$ but change the personality of
$v$ to paranoid.
\end{itemize}
\item If $v$ is paranoid, randomly recolour $v$ to a colour other than $\gamma$
(regardless of what colour $v$ has). Change the personality of $v$ to
thoughtful.
\end{itemize}

Note that for any vertex $v$, the personality of $v$ evolves according to a particular Markov chain indexed by the points in time when overtaking events happen. Indeed, paranoid vertices always transition to thoughtfulness, and thoughtful vertices stay thoughtful or transition to paranoid with probability 1/2. The initial personality distribution (in which we are paranoid with probability $1/3$) is precisely the stationary distribution of this Markov chain. Moreover, note that for any vertex $v$ and any time $t$, the personality of $v$ is independent from the entire history of the colours of its out-neighbours (and in particular, independent from the event that $\gamma$ overtakes for $v$ at time $t$, for any colour $\gamma$).

Now, recalling that $r$ is the root of our Galton--Watson tree, define the events
\begin{align*}
    \mathcal{P}_t & := \{r \text{ is paranoid at time }t\}\\
    \mathcal{T}_t & := \{r \text{ is thoughtful at time }t\}\\
    \mathcal{O}_{t}(\gamma) & := \{\gamma \text{ overtakes for }r \text{ at time }t\}\\
    \mathcal{C}_t & := \{r \text{ changes its colour just after time }t\},
\end{align*}
and let $f_t=\Pr[\mathcal C_t]$ be the probability that $r$ changes its colour just after time $t$ (i.e., during the $t$-th recolouring step). Then
\begin{align}
f_{t} & =\sum_{\gamma}\Pr[\mc O_t(\gamma)]\cdot\left(\vphantom\sum\Pr[\mc P_t]\cdot \Pr[\mc C_t\,|\,\mc P_t\cap \mc O_t(\gamma)]+\Pr[\mc T_t]\cdot\Pr[\mc C_t\,|\,\mc T_t\cap \mc O_t(\gamma)]\right)\notag\\
 & =\sum_{\gamma}\Pr[\mc O_t(\gamma)]\cdot\left(\frac13\cdot 1+\frac23\cdot \frac12\right)=\frac23 \sum_{\gamma}\Pr[\mc O_t(\gamma)]=2P_\lambda(f_{t-1}).\label{eq:ft-personality-changing-recurrence}
\end{align}

Crucially, this is the same recurrence as we na\"ively obtained in the last subsection, but with the factor of $3$ replaced with a factor of $2$. It essentially remains to study the function $P_\lambda$ and the initial change probability $f_1$; the following lemmas encapsulate the properties we will need.

\begin{lemma}\label{lem:f1}
For any $\lambda\ge 0$ we have $f_1\le 1/3$ .
\end{lemma}
\begin{lemma}\label{lem:P-slope}
For any $\lambda\ge0$ and $f\in[0,1/3]$ we have $2P_{\lambda}(f)\le2P_{\lambda}'(0) f\le 0.9999f$.
\end{lemma}
\begin{remark}\label{rem:7-9-11}
A simple way to prove that $2P_\lambda(f)<0.9999f$ for all $\lambda\ge 0$ would be to prove that $2Q_d(f)<0.9999f$ for all $d$ (since $P_\lambda$ is a weighted average of the $Q_d$). However, this is not true (in particular $2Q_d'(0)$ is slightly larger than 1 for $d\in \{7,9,11\}$), so we really need the averaging in the definition of $P_\lambda$.
\end{remark}
\cref{lem:f1} is more-or-less immediate: given any initial colouring of the out-neighbours of $r$, at most one of the three possible colours for $r$ would cause $r$ not to be majority-coloured. \cref{lem:P-slope} is more delicate, and we prove it in \cref{sec:computations} with computer assistance. We remark that the constant ``$0.9999$'' is not sharp, and is chosen merely for convenience of obtaining a rigorous proof (the best possible constant seems to be about $0.83$).

Now, \cref{lem:f1,lem:P-slope}, together with the recurrence in \cref{eq:ft-personality-changing-recurrence}, imply that $f_{t}\le (1/3)\cdot 0.9999^{t-1}<0.9999^{t}$, 
which is an upper bound on the probability that $r$ was majority-coloured at time $t$. We record this in the following lemma.
\begin{lemma}\label{lem:personality-changing-bound}
    For any $\lambda\ge 0$, consider a $\on{Poisson}(\lambda)$ Galton--Watson tree, with edges oriented away from the root, and consider the random recolouring process described in this section (with personality-changing). Just before the $t$-th recolouring step, the probability that the root is majority-coloured is at most $0.9999^t$. 
\end{lemma}

\cref{lem:local-output} is a near-immediate corollary, as follows.

\begin{proof}[Proof of \cref{lem:local-output}]
Let $\lambda=np$ and consider the random recolouring process (with personality-changing) on $D\sim \mb D(n,p)$.
Since this is a local process, by \cref{lem:GW,lem:personality-changing-bound}, each vertex $v$ fails to be majority-coloured at time $t$ with probability at most $0.9999^{t}+O((\log n)^{2t}/n)$. Taking say $t=\log \log n$, we see that the expected number of vertices which are not majority-coloured at time $t$ is $o(n)$, and the desired result follows by Markov's inequality.
\end{proof}

\begin{remark}\label{rem:high-girth}
    The above proof works for any sequence of digraphs (random or not) which locally converge to a $\on{Poisson}(\lambda)$ Galton--Watson tree. In fact, it is possible to make minor changes to the proof to handle arbitrary high-girth digraphs which do not have too many vertices that are close to each other and have out-degrees 7,9 or 11 (recall from \cref{rem:7-9-11} that $2Q_d'(0)$ can be greater than 1 if $d\in \{7,9,11\}$).
    
    For digraphs which have many nearby vertices with degree 7,9 and 11, we do not see how to obtain an analogue of \cref{lem:local-output} without obtaining some nontrivial bounds on conditional probabilities as in \cref{eq:conditional-probability}. However, since $2Q_d'(0)$ is only ever very slightly larger than 1, very weak bounds would suffice, and it may be possible to obtain such bounds by reasoning very carefully about how the likely colour history of a vertex relates to the colour histories of its out-neighbours.
\end{remark}

\section{A list-assignment process}\label{sec:marking}

In this section we prove \cref{lem:marking}.

Recall that \cref{lem:local-output} provides us with a 3-colouring of our random digraph
$D\sim \mb D(n,p)$ in which almost all vertices are majority-coloured. We prove
\cref{lem:marking} with a \emph{list-assignment process} that initially only assigns
lists to the non-majority-coloured vertices, but then recursively
considers the effects of choosing different colours in these lists.
Specifically, for a vertex $v$ with a list $L(v)$, if we were to
change the colour of $v$ (to some new colour in $L(v)$), we may
cause other vertices to become non-majority-coloured. Those vertices
which are in danger of becoming non-majority-coloured must themselves
be assigned lists, and the effects of their colour changes must be
recursively investigated. Mostly we assign lists of size 2, but occasionally
we need to assign lists of size 3, when a vertex could become non-majority-coloured
via two different pathways, or when we are in danger of violating
\cref{L4} or \cref{L5} of \cref{lem:marking}. 

In detail, our list-assignment process is defined as follows.
\begin{definition}\label{def:marking}
Fix a parameter $\ell\in\mb N$, a digraph $D$, and an initial
colouring $c:V(D)\to\{1,2,3\}$. Every vertex $v$ which has a list will always have a \emph{path danger level} $\on{pd}(v)\in\{0,1,\dots,\ell+1\}$. We say that
a vertex is \emph{defective} if it has a list with size 3. To ``make
a vertex defective'' is to give it the list $\{1,2,3\}$, and also to give it a path danger level of zero. At all points in time we write $U$ for the set of vertices which
have been assigned lists. 
\begin{enumerate}
\item Consider all the vertices which are not majority-coloured with respect to $c$. Make all such vertices defective. (We say these vertices are ``colour-defective''.)
\item Repeatedly do one of the following actions, as long as one is possible ((a) should always take first priority, and (c) should always take second priority, but otherwise, choose which action to do next according to some arbitrary but deterministic rule).
\begin{enumerate}
\item If some vertex $v$ has $\on{pd}(v)=\ell+1$, then make $v$ defective, changing $\on{pd}(v)$ to zero in the process. (We say $v$ is ``path-defective''.)
\item If there is a directed cycle of vertices in $U$ which currently has at most one defective vertex, then make all vertices in that cycle defective. (We say these vertices are ``cycle-defective'').
\item If there is any vertex $v$ (in all of $D$) which is an in-neighbour of two different vertices $u,u'\in U$, then make $v$ defective. (We say $v$ is ``duplicate-defective'').
\item If there is a vertex $u\in U$, a colour $\gamma\in L(u)$ and a vertex $v\notin U$, such that changing the colour of $u$ from $c(u)$ to $\gamma$ would cause $v$ to no longer be majority-coloured (this can only happen if $c(v)=\gamma$), then assign to $v$ the size-$2$ list $L(v)=\{c(v),c(v)+1\}$, where addition is mod 3. (This choice is basically arbitrary; the important thing is that $L(v)$ includes $c(v)$ and a second colour chosen according to some deterministic rule.) Moreover, set $\on{pd}(v)=\on{pd}(u)+1$.
\end{enumerate}
\end{enumerate}
\end{definition}

We emphasise that the above process is \emph{not} a random process (every step is deterministic), though we will only ever run it on random digraphs. Also, we emphasise that a vertex can be defective in ``two different ways'' (e.g., it is possible for a vertex to be both duplicate-defective and path-defective). 

After the list-assignment process completes, we have a list $L(v)$ assigned
to each vertex $v\in U$. By construction, these lists, together with
the colours $c(v)$ for each $v\notin U$, satisfy all the conditions
in \cref{lem:marking} except possibly \cref{L2} (assuming $\ell=O(1)$). So, in order
to prove \cref{lem:marking}, it suffices to show that if we run the random recolouring process described
in \cref{subsec:personality-changing} on a random digraph $D\sim \mb D(n,p)$, for some appropriate number of steps (to obtain an \emph{almost}-majority-colouring
$c:V(D)\to\{1,2,3\}$), and subsequently run the above list-assignment process
for some appropriate $\ell=O(1)$, then whp we end up with $|U|=o(n)$.

\subsection{Proof strategy}
In this subsection we state the two key lemmas that underpin the proof of \cref{lem:marking}. Recall from \cref{lem:local-output} that our random recolouring process whp provides us with a 3-colouring such that the set of non-majority-coloured vertices $U$ is very small. Roughly speaking, our approach to prove \cref{lem:marking} is as follows.
\begin{enumerate}
    \item[(A)] First, we show that, whp, \emph{if} we do not create too many duplicate-defective vertices then $U$ does not grow too much during the list-assignment process. This is because if we ignore duplicate-defective vertices then our list-assignment process is comparable to a subcritical branching process (assuming $\ell$ is sufficiently large that path-defective vertices do not play a major role).
    \item[(B)] Second, we show that whp the number of duplicate-defective vertices stays small relative to the number of vertices explored during the list-assignment process. Indeed, when we have explored a $\delta$-fraction of the graph, we expect about a $\delta^2$-fraction of vertices to have been seen more than once (thereby becoming duplicate-defective).
\end{enumerate}
That is to say, we bound the number of duplicate-defective vertices in terms of $|U|$, and conversely we bound $|U|$ in terms of the number of duplicate-defective vertices. At a very high level, this kind of ``self-bounding'' approach is common in the analysis of combinatorial random processes (perhaps most famously, in the \emph{differential equations method}; see \cite{Wor99DE}).

To formalise step (A), we define a ``virtual'' version of the list-assignment process, which ``replays a recording'' of some part of the list-assignment process, to investigate the knock-on effects that occur due to a particular set of vertices becoming duplicate-defective (or being initially colour-defective, or becoming cycle-defective, but neither of these types of defective vertices are too important as sparse random graphs typically have very few short cycles and our random recolouring process typically leaves very few colour-defective vertices).
\begin{definition}\label{def:virtual}
Fix a digraph $D$, a colouring $c:V(D)\to\{1,2,3\}$, a parameter $\ell\in\mb N$, a set of vertices $W$ and a sequence $\vec Q$ of elements of $W\cup (V(D)\times \mb N)$ containing each $w\in W$ exactly once (we call $\vec Q$ the ``tape''). 
The \emph{$(W,\vec Q)$-virtual} list-assignment process is defined as follows. We reuse the notation and terminology from \cref{def:marking}.
\begin{enumerate}
    \item Initially, no vertices have lists (i.e., $U=\emptyset$).
    \item While $\vec Q$ is nonempty: consider the first entry $e$ of $\vec Q$. 
    \begin{enumerate}
        \item If $e$ is a single vertex $w\in W$, then make $w$ defective, setting $\on{pd}(w)$ to zero in the process. (We say $w$ is ``virtual-defective''.)
        \item If $e$ is a pair $(u,i)$, then check if $u\in U$, and if $i\le \deg^-(u)$. If either of these does not hold, abort the entire process.
        \begin{itemize}
            \item Let $v$ be the $i$-th in-neighbour of $u$ (according to some pre-specified ordering of the in-neighbours of $v$).
            \item Check if $v$ is majority-coloured, and  changing $c(u)$ to some other colour $\gamma\in L(u)$ would cause $v$ to not be majority-coloured. Also, check if $v$ does not already have a list. If one of these conditions fails, abort the entire process.
            \item Set $\on{pd}(v)=\on{pd}(u)+1$.
            \item If $\on{pd}(v)=\ell+1$ then make $v$ defective. (We say $v$ is ``path-defective''.)
            \item Otherwise, assign to $v$ the size-2 list $L(v)=\{c(v),c(v)+1\}$, where addition is mod 3 (i.e., according to the same rule as in step (2d) in \cref{def:marking}).
        \end{itemize}
        \item Remove $e$ from $\vec Q$ (so the second element of $\vec Q$ becomes the first, and so on).
    \end{enumerate}
\end{enumerate}

Let $R(W,\vec Q)$ be the set of vertices which would be assigned a list
if we ran the $(W,\vec Q)$-virtual list-assignment process.
\end{definition}

The idea is that for every set of vertices $W$ that become defective at some point in the list-assignment process, there is some tape $\vec Q$ that records the order in which vertices were assigned lists as a result of the vertices in $W$ becoming defective (and as a result of the corresponding knock-on effects). The virtual list-assignment process takes $W$ and $\vec Q$ as input, and ``validates'' the tape (making sure that vertices could indeed have been assigned lists in that order).

We remark that the order in which vertices are processed can have quite a dramatic effect on the behaviour of the list-assignment process, purely due to the way path-defective vertices are defined (path-defective vertices occur ``every $\ell+1$ steps'', so if a vertex $v$ can be assigned a list via two different pathways of different lengths, whether or not $v$ is path-defective can depend on which pathway is taken first). This order-dependence is the reason we need a ``tape'' specifying the order in which vertices should be processed.

Now, steps (A) and (B) of our proof are captured in parts (A) and (B) of the following lemma.
\begin{lemma}\label{lem:defective-bounding}
Fix a constant $C\ge 0.1$. Let $0.1/n\le p\le C/n$, $t_{0}=(\log \log n)^2$, $\ell=10^{10}C$, and $D\sim\mb D(n,p)$. Run the random recolouring process (with personality-changing) described in \cref{subsec:personality-changing}, until time $t_0$, to obtain a colouring $c:V(D)\to \{1,2,3\}$. Then whp the following hold.
\begin{enumerate}
    \item [(A)] For every set $W$ of at least $n^{0.9}$ vertices, and any tape $\vec Q$, we have $|R(W,\vec Q)|\le2|W|(\log n)^{\ell+3}$.
    \item [(B)]Consider the list-assignment process described in \cref{def:marking} (which features an evolving set $U$ of list-assigned vertices). Let $\delta=1/(\log n)^{\ell+10}$; at every moment of the list-assignment process for which $|U|\le \delta n$,
the number of duplicate-defective vertices is at most $\delta^2 n (\log n)^3$.
\end{enumerate}
\end{lemma}

\begin{remark}
The choices of $t_0,\delta$ are fairly arbitrary, and the bounds in \cref{lem:defective-bounding} are rather crude. In particular, taking $t_0$ significantly larger than $\log \log n$ means that after $t_0$ steps of random recolouring, whp there are so few non-majority-coloured vertices that we can tolerate $U$ growing by any poly-logarithmic factor during the list-assignment process. This ``poly-logarithmic slack'' is very convenient as it allows us to take crude union bounds. With a more refined analysis, we expect that it should be possible to take $t_0,\ell$ to be sufficiently large constants and $\delta$ a sufficiently small constant, and it should be possible to remove logarithmic factors altogether.
\end{remark}

We need some preparations before proving \cref{lem:defective-bounding}.

\subsection{A marked configuration model}

In order to reveal the colouring arising from $t_{0}$ steps
of the random recolouring process, we must reveal certain information
about our random digraph $D\sim \mb D(n,p)$. Crucially, there is still plenty of
randomness remaining after this information is revealed; we need to use this randomness to study the list-assignment process. To get a
handle on the remaining randomness, we compare the conditional distribution
of $D$ to a ``marked configuration model'' (analogous to the well-known
\emph{configuration model} often used to study degree-constrained random graphs;
see for example \cite{Wor99}).

Basically, for each vertex $v$ we first reveal its in- and out-degrees $\deg^-(v)$ and $\deg^+(v)$ according to the distribution $\mb D(n,p)$ (together with certain information about the outcome of the random recolouring process). Then, we replace each vertex $v$ with $\deg^-(v)$ ``in-stubs'' and $\deg^+(v)$ ``out-stubs''; our marked configuration model is obtained by randomly matching out-stubs to in-stubs.
\begin{definition}\label{def:configuration-model}
A \emph{multidigraph} is the digraph analogue of a multigraph: it
may have directed loops, and it may have multiple edges going in the
same direction between a pair of vertices (called \emph{parallel edges}).
We say a multidigraph is \emph{marked} if each vertex has a sequence
of colours associated with it.

Also, in this definition we assume all (multi)digraphs have an ordering on their vertices (for example, the vertex set of $\mb D(n,p)$ can be taken to be $\{1,\dots,n\}$). For an ordered set $S$ of size $d$, and a function $\phi:S\to X$, we write $[\phi(s):s\in S]\in X^d$ to denote the sequence of values $\phi(s)$, according to the order of $S$ (note that this is a sequence in $X^d$, not a function in $X^S$; we ``forget'' the values of $S$ themselves).

\begin{itemize}
\item Let $\mc C(t)=\{1,2,3\}^t$. Let $\mb D_{t_{0}}(n,p)$ be the distribution of the random marked digraph
$(D,\vec{c})$ defined by taking $D\sim\mb D(n,p)$, running the random recolouring process described in \cref{subsec:personality-changing} (with personality-changing), until time $t_{0}$,
and for each vertex $v$ letting $\vec{c}(v)\in\mc C(t_0)$
be the sequence of colours taken by $v$ over the duration of the process.
\item For any marked (multi)digraph $(D,\vec{c})$ and any vertex $v\in V(D)$
let
\[T_{D,\vec{c}}(v)=(\deg^{+}(v),\deg^{-}(v),\vec{c}(v),[\vec{c}(u):u\in N^{+}(v)])\]
encode the in-degrees and out-degrees of $v$, the mark of $v$, and
the marks of the out-neighbours of $v$. We think of $T_{D,\vec{c}}$ as a function $V(D)\to \mb N\times \mb N\times \mc C(t_0)\times \bigcup_{i=0}^\infty\mc C(t_0)^{i}$. Let $\mathbb{T}_{t_{0}}(n,p)$ be the distribution
of $T_{D,\vec{c}}$, for $(D,\vec{c})\sim\mb D_{t_{0}}(n,p)$.
\item Let $\hat{\mb D}_{t_{0}}(n,p)$ be the distribution of the random marked
multidigraph $(\hat{D},\vec{c})$ defined as follows.
\begin{enumerate}
\item Consider $T\sim\mathbb{T}_{t_{0}}(n,p)$.
\item Let $V=\{1,\dots,n\}$ be the vertex set of $\mb D(n,p)$, and for each vertex $v\in V$:
\begin{enumerate}
\item Create $\deg^{+}(v)$ vertices called \emph{outgoing stubs} and $\deg^{-}(v)$
vertices called\emph{ incoming stubs} (we will always refer to these as ``stubs'', not ``vertices''). Here, $\deg^{+}(v)$ and $\deg^{-}(v)$
are as specified by $T$.
\item Mark the incoming stubs with the sequence $\vec{c}(v)$, and mark
the outgoing stubs with the sequences in $[\vec{c}(u):u\in N^{+}(v)]$
(in any order). Again, this data is as specified by~$T$.
\end{enumerate}
So, in total, we now have an empty graph with $\sum_{v\in V}\left(\deg^{+}(v)+\deg^{-}(v)\right)$
stubs, each of which is marked with a sequence of colours.
\item Then, for each of the $3^{t_{0}}$ sequences $\vec{\gamma}\in\mc C(t_0)$,
let $S_{\vec{\gamma}}^{+}$ be the collection of all outgoing stubs marked
with $\vec{\gamma}$ (among those generated by all vertices), and let $S_{\vec{\gamma}}^{-}$
be the collection of all incoming stubs marked with $\vec{\gamma}$. Note
that $|S_{\vec{\gamma}}^{+}|=|S_{\vec{\gamma}}^{-}|$ is the sum of in-degrees
of vertices in $D$ marked with $\vec{\gamma}$. 
\item For each $\vec \gamma\in \mc C(t_0)$, consider a uniformly random perfect matching (independent from the
remaining randomness of $D$) between $S_{\vec{\gamma}}^{-}$ and $S_{\vec{\gamma}}^{+}$,
and orient each edge of this matching from $S_{\vec{\gamma}}^{+}$ to $S_{\vec{\gamma}}^{-}$.
\item Now, for each vertex $v$, consider the $\deg^{+}(v)+\deg^{-}(v)$ stubs
that arose from $v$, and contract these stubs back to a single vertex
$v$. This gives a multidigraph $\hat{D}$ with $T_{\hat{D},\vec c}=T$.
\end{enumerate}
\end{itemize}
\end{definition}

Now, if $(D,\vec{c})\sim\mb D_{t_{0}}(n,p)$ and we condition on an
outcome of $T_{D,\vec c}$, then by symmetry $D$ is simply a uniformly random digraph
consistent with this $T_{D,\vec c}$. Also, if $(\hat{D},\vec{c})\sim\hat{\mb D}_{t_{0}}(n,p)$
and we condition on an outcome of $T_{\hat D,\vec c}$, then each possible
outcome of $D$ is equally likely to appear as $\hat{D}$ (in (4), the number of matchings which would yield $D$ is always exactly
$\prod_{v}\deg^{+}(v)!\deg^{-}(v)!$). However, $\hat{D}$ can also take
outcomes which are impossible for $D$ (namely, those outcomes with
loops or parallel edges). 
We record these observations as follows.
\begin{fact}\label{fact:configuration-model-conditioning}
Let $(D,\vec{c})\sim\mb D_{t_{0}}(n,p)$ and $(\hat{D},\vec{c})\sim\hat{\mb D}_{t_{0}}(n,p)$,
and consider any possible outcome $T$ of $\mathbb{T}_{t_{0}}(n,p)$.
Then the conditional distribution of $D$ given $T_{D,\vec c}=T$ is the
same as the conditional distribution of $\hat{D}$ given that $T_{\hat{D},\vec c}=T$
and that $\hat{D}$ has no loops and no parallel edges.
\end{fact}

\subsection{Preliminary lemmas on random marked digraphs}

Before going further, we state some properties of $\mb D_{t_{0}}(n,p)$ and $\hat{\mb D}_{t_{0}}(n,p)$. One of these is that whp $T\sim \mb T_{t_0}(n,p)$ has statistics very closely approximated by consideration of a Galton-Watson tree; to state this formally we need a definition.
\begin{definition}\label{def:mu-nu}
Consider a $\operatorname{Poisson}(\lambda)$
Galton--Watson tree with root $r$ (with edges oriented away from $r$). Independently, add a $\on{Poisson}(\lambda)$ number of in-neighbours to $r$. On this tree, run $t_{0}$ steps of the random recolouring process described in \cref{subsec:personality-changing}, thereby obtaining a colour history sequence $\vec c(v)\in \mc C(t_0)$ for each vertex $v$. For any $R=(d^{+},d^{-},\vec{\gamma},(\vec{\gamma}(1),\dots,\vec{\gamma}(d^+)))$, let  
\begin{align*}
    \mu_{t_0,\lambda}(R)=\Pr[(\deg^{+}(r),\deg^{-}(r),\vec{c}(r),[\vec{c}(u):u\in N^{+}(r)])=R],\qquad
    \nu_{t_0,\lambda}(\vec{\gamma})=\Pr[\vec{c}(r)=\vec \gamma].
\end{align*}
We omit the subscripts $t_0,\lambda$ when they are clear from context.
\end{definition}

Now, the following basic properties of $\mb D_{t_0}(n,p)$ follow from routine calculations.

\begin{lemma}\label{lem:marking-prelim}
Fix constants $C>0$ and $\ell\in \mb N$, let $p=\lambda/n$ for some $\lambda\le C$, let $t_0=(\log \log n)^2$, and let $(D,\vec c)\sim\mathbb{D}_{t_{0}}(n,p)$ (so $T_{D,\vec c}\sim \mb T_{t_0}(n,p)$).
Then, whp the following properties
are satisfied.
\begin{enumerate}[{\bfseries{D\arabic{enumi}}}]
\item\label{D1}
$\deg^{+}(v),\deg^{-}(v)\le\log n$ for each $v$.
\item\label{D2} For every $R\in \mb N\times \mb N\times \mc C(t_0)\times \bigcup_{i=0}^\infty\mc C(t_0)^{i}$, the number of vertices $v$ with $T_{D,\vec c}(v)=R$ is $\mu(R) n+O(n^{2/3})$.
\item \label{D3} There are at most $\sqrt n$ vertices in cycles of length at most $2\ell$.
\end{enumerate}
\end{lemma}
\begin{proof}[Proof sketch]
\cref{D1} is a routine consequence of the Chernoff bound, and \cref{D3} follows from Markov's inequality and the fact that the expected number of cycles of length at most $2\ell$ is at most
\[\sum_{i=2}^{2\ell} n^i p^i\le \sum_{i=2}^{2\ell} C^i=O(1).\]

For \cref{D2}, recall the random recolouring process described in \cref{subsec:personality-changing}, which describes how a vertex $v$ should change its colour when an overtaking event occurs, depending on the personality of $v$ and potentially the result of a coin flip. We imagine that
each vertex $v$ has a sequence of coin flips $s(v)\in \{0,1\}^{t_0}$ (the first two of which are biased to land heads with probability $1/3$ and $1/4$, and the rest of which are unbiased, landing heads with probability $1/2$). The first ($1/3$-biased) coin flip can be used to decide the initial personality of $v$, and the subsequent coin flips can be used to determine which colour $v$ should change to at each overtaking event. (In the paranoid case, we just need an unbiased coin flip; in the thoughtful case we need a coin flip of bias $(1/2-p^*)/(1-p^*)$, which is equal to $1/4$, $1/2$ or $0$ when $p^*$ is equal to $1/3$, $0$ or $1/2$ respectively). 

Also, we slightly modify the random recolouring process: vertices with in-degree greater than $\log n$ always have the colour ``1'' (i.e., they do not randomly change their colour). When \cref{D1} occurs (which it does whp), this change makes no difference to the process.

Let $X_R$ be the number of vertices $v$ with $T_{D,\vec c}(v)=R$ in our modified process. The purpose of our modification is that if we consider a vertex $v$, and we consider any change to the initial colour of $v$, or $s(v)$, or the set of edges which are incident to $v$, then $X_R$ changes by at most $(\log n)^{t_0}=n^{o(1)}$ (because the number of vertices that can be affected by our single-vertex change grows by a factor of at most $\log n$ in every step of the process). So, the desired result follows from the Azuma--Hoeffding inequality (see for example \cite[Theorem~7.2.1]{AS16}).
\end{proof}
We also need some consequences of \cref{D1,D2} above. To state these we need some further definitions.
\begin{definition}
Given a marked digraph $(D,\vec c)$, let $V_{\vec \gamma}$ be the set of vertices marked with $\vec \gamma$. Let $\deg^{\vec \gamma}(v)$ be the number of out-neighbours of $v$ marked with $\vec \gamma$. So,
\[\sum_{v\in V}\deg^{\vec \gamma}(v)=\sum_{v\in V_{\vec \gamma}}\deg^{-}(v).\]   

Also, say that a vertex $v$ is \emph{$\gamma$-critical}
if $c_{t_{0}}(v)=\gamma$ and if $v$ has exactly $\lfloor \deg^{+}(v)/2\rfloor$
out-neighbours $w$ with $c_{t_{0}}(w)=\gamma$ (i.e., if $\gamma$
is ``almost'' a majority colour among the out-neighbours of $v$,
with respect to $c_{t_{0}}$). Let $\mr{Crit}(\gamma)$ be the set of vertices
which are $\gamma$-critical.
\end{definition}
We next observe that the probabilities $\nu(\vec{\gamma})$ defined in \cref{def:mu-nu} are not too small (this is convenient for concentration inequalities).

\begin{lemma}\label{lem:GW-probabilities-big}
Fix a constant $C\ge 0.1$ and let $0.1\le \lambda \le C$. For every $\vec \gamma\in \mc C(t_0)$, we have $\nu_{t_0,\lambda}(\vec{\gamma})\ge \exp(-O(t_0^2))$.
\end{lemma}
\begin{proof}
    Consider a $\on{Poisson}(\lambda)$ Galton--Watson tree with root $r$ (with edges oriented away from $r$). 
    Say that an index $t$ is a \emph{flip} if $\gamma_t\ne \gamma_{t+1}$. Note that we can greedily find a sequence of colours $\gamma'_1,\dots,\gamma'_{t_0}$ such that $\gamma'_t\ne \gamma'_{t+1}$ for each $t$, and $\gamma_t=\gamma_t'$ if and only if $t$ is a flip. Indeed, for $t=1,\dots,t_0$ (in that order): if $t$ is not a flip then there is at least one choice for $\gamma_t'$ which is different from $\gamma_{t-1}'$ and $\gamma_t$, and if  $t$ is a flip, then we can set $\gamma_t'=\gamma_{t}$ (by considering the cases where $t-1$ is a flip or not, one can check that this colour is different to $\gamma_{t-1}'$).
    
    Now, consider the following potential sequence of events.
    \begin{enumerate}
        \item Regarding the structure of the tree itself: in each of the first $t_0-1$ generations, only a single child is born (so the local structure is a path of length $t_0-1$ away from $r$). For $\ell< t_0$, let $v_\ell$ be the unique vertex at distance $\ell$ from $v$. 
        \item The initial colour of $r=v_0$ (at time $1$) is $\gamma_1$.
        \item The initial colours of $v_1,\dots,v_{t_0-1}$ are all $\gamma_1'$. 
        \item At the $(t-1)$st recolouring step (i.e., at the $t$-th step including the initial colouring):
        \begin{itemize}
            \item $r$ changes its colour to $\gamma_{t}$ (if it already has colour $\gamma_t$, it does not change).
            \item Each of $v_1,\dots,v_{t_0-t}$ change their colour to $\gamma_{t}'$. 
        \end{itemize} 
    \end{enumerate}
    Note that this sequence of events occurs with probability $\exp(-\Theta(t_0^2))$. Indeed, the probability that the first $t_0$ generations have the desired path-like structure is $\exp(-\Theta(t_0))$ (here we are using that $\lambda\ge 0.1$). Then, in the recolouring process, the events at the $t$-th step occur with probability $\exp(-\Theta(t_0-t))$.
    
    Finally, if this sequence of events occurs, note that $r$ has colour history $\vec \gamma$.
\end{proof}

We next state some consequences of \cref{D1,D2}.
\begin{lemma}\label{marking-prelim-consequences}
    Recall the notation and definitions in \cref{lem:marking-prelim}, and assume $\lambda\ge 0.1$. Consider an outcome of $T_{D,\vec c}$  of $\mb T_{t_0}(n,p)$ such that \cref{D1,D2} hold. Then the following further properties hold.
    \begin{enumerate}[{\bfseries{T\arabic{enumi}}}]
    \item \label{T3} For every $d\in \mb N$ and every $(\vec{\gamma}(1),\dots,\vec{\gamma}(d))\in \mc C(t_0)^d$,
there are 
\[
\left(\frac{e^{-\lambda}\lambda^{d}}{d!}\prod_{i=1}^{d}\nu(\vec \gamma(i))\right)n+O(n^{3/4})
\]
vertices $v$ for which $[\vec{c}(u):u\in N^{+}(v)]=(\vec{\gamma}(1),\dots,\vec{\gamma}(d))$.
    \item\label{T1} There are at most $0.9999^{t_0}n+n^{3/4}$  vertices which are not majority-coloured with respect to $c_{t_0}$.
    \item \label{T-stub-count} For every $\vec{\gamma}\in\mc C(t_0)$ we have
    \[\sum_{v\in V}\deg^{\vec \gamma}(v)=\sum_{v\in V_{\vec \gamma}}\deg^{-}(v)=\lambda\nu(\vec{\gamma})n+O(n^{3/4}).\]
    \item \label{T-second-moment} For every $\vec{\gamma}\in\mc C(t_0)$ we have
    \[\sum_{v\in V}\deg^{\vec \gamma}(v)^2=O(\nu(\vec{\gamma}) n),\quad \sum_{v\in V_{\vec \gamma}}\deg^{-}(v)^2=O(\nu(\vec{\gamma}) n),\quad \sum_{v\in V\cap V_{\vec\gamma}}\deg^{-}(v)\deg^{\vec \gamma}(v)=O(\nu(\vec{\gamma}) n).\]
\item\label{T-critical} For any $d\in \mb N$,
\[
\sum_{\substack{v\in V:\\\deg^{-}(v)=d}}\deg^{\vec \gamma}(v)= \frac{e^{-\lambda}\lambda^{d}}{d!}\sum_{v\in V}\deg^{\vec \gamma}(v)+O(n^{3/4}).\]
Moreover, for any $\gamma\in\{1,2,3\}$ with $\gamma_{t_{0}}\ne\gamma$, we have
\[
\sum_{\substack{v\in \mr{Crit}(\gamma):\\\deg^{-}(v)=d}}\deg^{\vec \gamma}(v)\le \frac{0.99999}{\lambda}\cdot\frac{e^{-\lambda}\lambda^{d}}{d!}\sum_{v\in V}\deg^{\vec \gamma}(v)+O(n^{3/4}).\]
\end{enumerate}
\end{lemma}
\begin{proof}
As in \cref{def:mu-nu}, consider a $\on{Poisson}(\lambda)$ Galton--Watson
tree with root $r$, with edges oriented away from $r$, and with
a $\on{Poisson}(\lambda)$ number of in-neighbours added to $r$.
Consider $t_{0}$ steps of our random recolouring process on this
tree. (Recalling \cref{D2}, we can study this Galton--Watson tree to deduce statistical information about $T_{D,\vec c}$).

First, \cref{T3} follows basically immediately from the definition of $\mu(R)$,
and the fact that the children of $r$ have independent colour histories
(since the random recolouring process never looks at in-neighbours).
Note that we need to add the contributions from $3^{t_{0}}\log n$
different $R$ (corresponding to each of the possibilities for $\deg^{-}(v)$
and $\vec{c}(v)$), so the corresponding error terms in \cref{D2} must be
compounded.

For \cref{T1}, recall from \cref{lem:personality-changing-bound} that after $t_{0}$ steps, $r$
is majority-coloured with probability at least $0.9999^{t_{0}}$.
We can then add the contributions from all $R$ such that $r$ is
majority-coloured at time $t_{0}$.

Third, \cref{T-stub-count} similarly follows from the facts that $\mb E[ \deg^{+}(r)]=\lambda$,
and that each child of $r$ has colour history $\vec{\gamma}$ with
probability $\nu(\vec{\gamma})$, independently (we then need to consider
a \emph{weighted} sum of contributions from different $R$, where
the weights are at most $\log n$ by \cref{D1}).

\cref{T-second-moment} is very similar to \cref{T-stub-count}, except that we instead use the formulas (writing $\mbm 1\{A\}$ for the indicator random variable of an event $A$)
\begin{align*}
\mb E\big[\deg^{\vec \gamma}(r)^{2}\big] & =\mb E\big[\deg^{+}(r)(\deg^{+}(r)-1)\big]\nu(\vec{\gamma})^{2}+\mb E\big[\deg^{+}(r)\big]\nu(\vec{\gamma})=(\lambda\nu(\vec{\gamma}))^{2}+\lambda\nu(\vec{\gamma}),\\
\mb E\big[\one\{\vec{c}(r)=\vec{\gamma}\}\,\deg^{-}(r)^{2}\big] & =\nu(\vec{\gamma})\mb E\big[(\deg^{-}(r))^{2}n\big]=\nu(\vec{\gamma})(\lambda^{2}+\lambda),\\
\mb E\big[\one\{\vec{c}(r)=\vec{\gamma}\}\,\deg^{-}(r)\deg^{\vec \gamma}(r)\big] & =\nu(\vec{\gamma})\mb E\big[\deg^{-}(r)\big]\mb E\big[\deg^{+}(r)\big]\nu(\vec{\gamma})=(\lambda\nu(\vec{\gamma}))^{2}.
\end{align*}
(All of these expressions are of the form $O(\nu(\vec{\gamma}))$, viewing $\lambda$ as a constant).

The first part of \cref{T-critical} again follows similarly, using that
\[\mb E[\one\{\deg^-(v)=d\}\,\deg^{\vec \gamma}(r)]=\frac{e^{-\lambda}\lambda^{d}}{d!}\mb E[\deg^{\vec \gamma}(r)].\]
The second part of \cref{T-critical} is morally similar, but the calculations are a bit involved. 
If we condition on $r$ having exactly $d$ out-neighbours,
then, for all $i$, the probability that $r$ has $\floor{d/2}$ out-neighbours $w$
with $c_{t_{0}}(w)=\gamma$, and that the $i$-th out-neighbour has
colour history $\vec{\gamma}$, is
\[
\nu(\vec{\gamma})\binom{d-1}{\floor{d/2}}\left(\frac{1}{3}\right)^{\floor{d/2}}\left(\frac{2}{3}\right)^{d-1-\floor{d/2}}.
\]
Recall that (due to the personality-changing) the event $c_{t_{0}}(r)=\gamma$
occurs with conditional probability at most $2/3$ after conditioning
on any outcome of the colour histories of the out-neighbours of $r$.
So,
\begin{align*}
\mb E[\one\{r\text{ is }\gamma\text{-critical}\}\,\deg^{\vec \gamma}(r)] & \le\frac{2}{3}\nu(\vec{\gamma})\sum_{d=0}^{\infty}\frac{\lambda^{d}e^{-\lambda}}{d!}\cdot d\cdot \binom{d-1}{\floor{d/2}}\left(\frac{1}{3}\right)^{\floor{d/2}}\left(\frac{2}{3}\right)^{d-1-\floor{d/2}}
\\
 & =\nu(\vec{\gamma})\sum_{d=0}^{\infty}\frac{\lambda^{d}e^{-\lambda}}{d!}\left(d-\floor{d/2}\right)\binom{d}{\floor{d/2}}\left(\frac{1}{3}\right)^{\floor{d/2}}\left(\frac{2}{3}\right)^{d-\floor{d/2}}.
\end{align*}
Now, recall the polynomial $Q_{d}$ and the function $P_{\lambda}$
from \cref{sec:processes}. We compute 
\[
Q_{d}'(0)=\frac12 \left(d-\floor{d/2}\right)\binom{d}{\floor{d/2}}\left(\frac{1}{3}\right)^{\floor{d/2}}\left(\frac{2}{3}\right)^{d-\floor{d/2}}.
\]
(Note that in the formula in \cref{eq:Qd}, one only needs to consider the term
with $i=\floor{d/2}$, $j=\floor{d/2}+1$ and $k=0$, in which $f$ appears with a power of 1).
So, by \cref{lem:P-slope},
we have
\[
\mb E[\one\{r\text{ is }\gamma\text{-critical}\}\,\deg^{\vec \gamma}(r)] \le\nu(\vec{\gamma})\sum_{d=0}^{\infty}\frac{\lambda^{d}e^{-\lambda}}{d!}\cdot 2Q_{d}'(0) \le\nu(\vec{\gamma})\cdot 2P_{\lambda}'(0)\le 0.9999\nu(\vec \gamma)=\frac{0.9999}{\lambda}\mb E[\deg^{\vec \gamma}(r)],
\]
and 
\[
\mb E[\one\{\deg^-(v)=d\text{ and }r\text{ is }\gamma\text{-critical}\}\,\deg^{\vec \gamma}(r)]\le \frac{e^{-\lambda}\lambda^{d}}{d!}\cdot\frac{0.9999}{\lambda}\cdot \mb E[\deg^{\vec \gamma}(r)].
\]
We can then prove \cref{T-critical} by considering an appropriate weighted sum of
contributions from different $R$.
\end{proof}
Next, the following lemma shows that $\mb D_{t_{0}}(n,p)$ and $\hat{\mb D}_{t_{0}}(n,p)$ are very closely related (cf.\ ``contiguity'' lemmas in the study of random regular graphs; see \cite[Section~4]{Wor99}).
\begin{lemma}\label{lem:probably-simple}
Fix a constant $C$, let $p\le C/n$, let $(D,\vec{c})\sim\mb D_{t_{0}}(n,p)$ and let $(\hat{D},\vec{c})\sim\hat{\mb D}_{t_{0}}(n,p)$. Consider any possible outcome $T$ of $\mb T_{t_0}(n,p)$, such that \cref{D1,D2} hold.

For any $q\in[0,1]$, if an event holds with probability at least $1-q$ for $\hat D$, conditioned on the event $T_{\hat D,\vec c}=T$, then it holds with probability at least $1-\exp(O(3^{t_0}))q$ for $D$, conditioned on the event $T_{D,\vec c}=T$.
\end{lemma}
\begin{proof}[Proof sketch]
Recall the definition of $\hat{D}\sim\hat{\mb D}_{t_{0}}(n,p)$ via
random matchings between pairs of sets $S_{\vec{\gamma}}^{-},S_{\vec{\gamma}}^{+}$.
For a matching $M_{\vec{\gamma}}$ from $S_{\vec{\gamma}}^{+}$ to
$S_{\vec{\gamma}}^{-}$, say that an edge of $M_{\vec{\gamma}}$ is
\emph{loop-inducing} if it goes from an outgoing stub to an incoming
stub corresponding to the same vertex, and say that a pair of edges
of $M_{\vec{\gamma}}$ are \emph{parallel-inducing} if they go from outgoing
stubs corresponding to a common vertex, to incoming stubs corresponding
to a second common vertex. Say that $M_{\vec{\gamma}}$ is \emph{simple-inducing}
if it has no loop-inducing edge, and no pair of parallel-inducing
edges.

Note that conditioning on the event that $\hat{D}$ has no loops or
parallel edges is precisely the same as conditioning on the event
that each $M_{\vec{\gamma}}$ is simple-inducing. The reason we don't
need to take into account the interaction between different $M_{\vec{\gamma}}$
is that for each vertex $v$, all the incoming stubs are marked with
a common colour sequence $\vec{\gamma}$ (so the only possible parallel
edges pointing towards $v$ arise from $M_{\vec{\gamma}}$).

Recalling that the matchings $M_{\vec \gamma}$ are independent, it suffices to prove that for each $\vec \gamma\in\mc C(t_0)$, the random matching $M_{\vec{\gamma}}$ is simple-inducing with probability $\Omega(1)$. Similar statements have been proved many times for various types of random graph models with various assumptions (see \cite{Jan14} and the references therein), and the standard techniques all work here. Perhaps the simplest way to deduce our desired fact from results in the literature is to observe that a random simple-inducing matching $M_{\vec{\gamma}}$ is equivalent to a random (simple) \emph{bipartite} graph (with ``left-vertices'' and ``right-vertices'') where:
\begin{itemize}
    \item the left-vertices are copies of the vertices in $V_{\vec \gamma}$;
    \item the right-vertices are copies of the vertices of $D$ which have least one out-neighbour with colour history $\vec \gamma$;
    \item the left-degrees are constrained to be $\deg^-(v)$;
    \item the right-degrees are constrained to be $\deg^{\vec \gamma} (v)$;
    \item we forbid any edge between two copies of the same vertex.
\end{itemize}
The number of viable bipartite graphs can be approximated very accurately using the formula in \cite[Theorem~2.3(b)]{McK10} (it involves the quantities $\sum_{v\in V}\deg^{\vec \gamma}(v)^2\sum_{v\in V_{\vec \gamma}}\deg^{-}(v)^2$ and $\sum_{v\in V\cap V_{\vec\gamma}}\deg^{-}(v)\deg^{\vec \gamma}(v)$, which roughly correspond to the expected numbers of parallel edges and loops). The desired estimate then follows from \cref{T-second-moment} (we divide the number of viable bipartite graphs by the total number $|S^+_{\vec{\gamma}}|!$ of possibilities for $M_{\vec \gamma}$).
\end{proof}

\subsection{Bounding the duplicate-defective vertices}
In this subsection we prove \cref{lem:defective-bounding}(B).

\begin{proof}[Proof of \cref{lem:defective-bounding}(B)]
Let $(D,\vec{c})\sim\mb D_{t_{0}}(n,p)$. Recall that our goal is to prove that if we run the list-assignment
process on $D$ (with initial colouring $c=c_{t_0}$), then whp, at every
moment with $|U|\le \delta n$, the number of duplicate-defective vertices
is at most $\delta^2 n(\log n)^3$.

Let $(\hat{D},\vec{c})\sim\hat{\mb D}_{t_{0}}(n,p)$, and for the rest of the proof condition on an outcome of $T_{\hat D,\vec c}$ satisfying \cref{D1,D2}. By \cref{lem:marking-prelim} and \cref{lem:probably-simple}, it suffices to prove that in our conditional probability space, with probability at least, say, $1-\exp(-n^{0.1})$, at every
moment with $|U|\le\delta n$, the number of duplicate-defective vertices
is at most $\delta^2 n(\log n)^3$.

\medskip
\noindent\textit{Step 1: Setup for iterative exposure.} Recall from \cref{def:configuration-model} that
$\hat{\mb D}_{t_{0}}(n,p)$ is defined via a uniformly random matching
$M_{\vec{\gamma}}$ between each $S_{\vec{\gamma}}^{-}$ and $S_{\vec{\gamma}}^{+}$.
We say stubs $s^{-}\in S_{\vec{\gamma}}^{-}$ and $s^{+}\in S_{\vec{\gamma}}^{+}$
are \emph{partners} if they form an edge in $M_{\vec{\gamma}}$. Note
that we can reveal $M_{\vec{\gamma}}$ gradually, by repeatedly choosing
an incoming stub $s^{-}\in S_{\vec{\gamma}}^{-}$ whose partner has
not yet been revealed, and revealing the partner $s^{+}\in S_{\vec{\gamma}}^{+}$
of $s^{-}$. No matter how we choose the order in which incoming stubs have their
partners revealed (even if we choose this order adaptively,
based on the outcomes of previous revelations), the order in which
\emph{outgoing} stubs have their partners revealed is uniformly random.

The upshot of this observation is that we can define $M_{\vec{\gamma}}$
in terms of a uniformly random ordering $\prec_{\vec{\gamma}}$ of
each $S_{\vec{\gamma}}^{+}$, together with a rule to decide in which
order to reveal partners of stubs in $S_{\vec{\gamma}}^{-}$. The
order we choose is precisely the order in which vertices are processed
in the list-assignment process. Another way to say this is: we run
the list-assignment process on $\hat{D}$ while gradually revealing
the necessary information about $\hat{D}$. When we process a vertex
$v$, we need to examine its in-neighbours (in order to see whether they need to be assigned lists in response to the list of $v$), which amounts to, for
each $\vec{\gamma}$, revealing the partners of each of the incoming
stubs in $S_{\vec{\gamma}}^{-}$ associated with $v$. We simply choose
these partners to be the next available stubs in $S_{\vec{\gamma}}^{+}$,
according to the ordering $\prec_{\vec{\gamma}}$.

\medskip
\noindent\textit{Step 2: Characterising duplicate-defective vertices.} The
idea now is to describe the duplicate-defective vertices in terms
of the random orderings $\prec_{\vec{\gamma}}$. Suppose we generate
the $M_{\vec{\gamma}}$ in the above way, and suppose that at some
point in the process, we have revealed the partners of exactly $a_{\vec{\gamma}}$
of the stubs in $S_{\vec{\gamma}}^{+}$ (for each $\vec{\gamma}\in\mc C(t_0)$).
By definition, these partner-revealed stubs are precisely the first
$a_{\vec{\gamma}}$ stubs in $S_{\vec{\gamma}}^{+}$ with respect
to $\prec_{\vec{\gamma}}$ (write $S_{\vec{\gamma}}^{+}(a_{\vec{\gamma}})$
for the set of these stubs). Now, a vertex can have been revealed as duplicate-defective only if it has two different outgoing stubs among the $S_{\vec{\gamma}}^{+}(a_{\vec{\gamma}})$.

For a vector $\vec{a}=(a_{\vec{\gamma}}:\vec{\gamma}\in\mc C(t_0))$,
let $X(\vec{a})$ be the number of vertices which have two different
outgoing stubs among the $S_{\vec{\gamma}}^{+}(a_{\vec{\gamma}})$.
By \cref{D1}, at each point in the list-assignment process, the
total number of revealed edges in the matchings $M_{\vec\gamma}$ is at most $|U|\log n$. So,
it suffices to prove that with probability at least
$1-\exp(-n^{0.1})$, for every choice of $\vec{a}$ with $\|\vec{a}\|_{1}\le\delta n\log n$,
we have $X(\vec{a})\le\delta^2 n(\log n)^3$.

The number of choices of $\vec a$ is at most $(\delta n \log n)^{3^{t_0}}=\exp(o(n^{0.1}))$, so by the union
bound, it actually suffices to individually show that for each $\vec{a}$
with $\|\vec{a}\|_{1}\le\delta n\log n$, we have $X(\vec{a})\le\delta^2 n(\log n)^3$
with probability at least, say, $1-\exp(-2n^{0.1})$.

\medskip
\noindent\textit{Step 3: Expected value computation.} We now compute $\mb E X(\vec{a})$.
For $i<j\le\log n$, let $\mathcal{E}(v,i,j)$ be the event that a
vertex $v$ contributes to $X(\vec{a})$ via its $i$-th and $j$-th
out-neighbours (i.e., the $i$-th and $j$-th outgoing stubs corresponding
to $v$ are both among the $V_{\vec{\gamma}}^{+}(a_{\vec{\gamma}})$). Let $\mc E(v)=\bigcup_{i,j}\mc E(v,i,j)$ be the event that $v$ contributes to $X(\vec{a})$ via any pair of out-neighbours.

Suppose $v$ has out-degree at least $j$, its $i$-th out-neighbour
has colour history $\vec{\gamma}$ and its $j$-th out-neighbour has
colour history $\vec{\eta}$. Then, using \cref{T-stub-count}, we have
\begin{align*}
\Pr[\mathcal{E}(v,i,j)] & =\begin{cases}
\displaystyle{\frac{a_{\vec{\gamma}}}{|S_{\vec{\gamma}}^{+}|}\cdot\frac{a_{\vec{\eta}}}{|S_{\vec{\eta}}^{+}|}} & \text{if }\vec{\gamma}\ne\vec{\eta},\\
\displaystyle{\frac{a_{\vec{\gamma}}(a_{\vec{\gamma}}-1)}{|S_{\vec{\gamma}}^{+}|(|S_{\vec{\gamma}}^{+}|-1)}} & \text{if }\vec{\gamma}=\vec{\eta}
\end{cases}\\
 & \le O\left(\frac{a_{\vec{\gamma}}a_{\vec{\eta}}}{\nu(\vec{\gamma})\nu(\vec{\eta})n^{2}}\right).
\end{align*}

By \cref{T3}, for a \emph{random} vertex $v_{\mr{rand}}$, the distribution
of $[\vec{c}(u):u\in N^{+}(v_{\mr{rand}})]$ is the same, up to
total variation distance $O(n^{-{1/4}})$, as the distribution of the random sequence
$(\vec{\gamma}(1),\dots,\vec{\gamma}(d))$ obtained by first letting
$d\sim\on{Poisson}(\lambda)$, and then letting $\vec{\gamma}(1),\dots,\vec{\gamma}(d)\in\mathcal{C}(t_{0})$
be independent random sequences each with distribution given by $\nu$.
So, we have
\[
\Pr[\mathcal{E}(v_{\mr{rand}},i,j)]=\Pr[\on{Poisson}(\lambda)\ge j]\sum_{\vec{\gamma},\vec{\eta}\in\mathcal{C}(t_{0})}\nu(\vec{\gamma})\nu(\vec{\eta})\cdot O\left(\frac{a_{\vec{\gamma}}a_{\vec{\eta}}}{\nu(\vec{\gamma})\nu(\vec{\eta})n^{2}}\right)+O(n^{-1/4}).
\]
Note that for $Z\sim \on{Poisson}(\lambda)$ we have
\[\sum_{i,j:1\le i<j}\Pr[\on{Poisson}(\lambda)\ge j]=\sum_{q=0}^\infty q\,\Pr[Z>q]=\frac{\mb E Z^2}2=O(1),\]
so
$\Pr[\mathcal E(v_{\mr{rand}})]\le \sum_{\vec{\gamma},\vec{\eta}\in\mathcal{C}(t_{0})} O(a_{\vec{\gamma}}a_{\vec{\eta}}/n^{2})$.
It follows that
\[
\mb E X(\vec{a})\le\sum_{\vec{\gamma},\vec{\eta}\in\mathcal{C}(t_{0})}n\cdot O\left(\frac{a_{\vec{\gamma}}a_{\vec{\eta}}}{n^{2}}\right)=O\left(\frac{1}{n}\|\vec{a}\|_{1}^{2}\right)\le O(\delta^{2}n(\log n)^{2}).
\]

\medskip
\noindent\textit{Step 4: Concentration.} Now, note that we can generate the
random orderings $\prec_{\vec{\gamma}}$ as follows. Associate with
each outgoing stub $s\in\bigcup_{\vec{\gamma}\in\mathcal{C}(t_{0})}S_{\vec{\gamma}}^{+}$
an independent $\on{Uniform}([0,1])$ random variable $Z_{s}$. With
probability 1, all of these random variables are distinct. Then,
for each $\vec{\gamma}\in\mathcal{C}(t_{0})$, let $\prec_{\vec{\gamma}}$ be the ordering of $S_{\vec{\gamma}}^{+}$ defined by the relative sizes of the corresponding $Z_{s}$.
Note that if we alter any individual $Z_{s}$, we cannot change $X(\vec{a})$ by more than 1. So, by the Azuma--Hoeffding inequality (see for example \cite[Theorem~7.2.1]{AS16}), we have
\[
\Pr[X(\vec{a})\ge\delta^{2}n(\log n)^{3}]\le\exp\left(-\frac{\big(\delta^{2}n(\log n)^{3}-O(\delta^{2}n(\log n)^{2})\big)^2}{O(n)}\right)\le\exp(-2n^{0.1}),
\]
as desired.
\end{proof}
\begin{remark}\label{rem:defective-general}
    In the statement of \cref{lem:defective-bounding}(B), we specified $\delta=1/(\log n)^{\ell+10}$, but in the above proof, we really only used that $\delta$ is ``not too big'' (say, $\delta\le n^{-0.1}$ would have sufficed). To be precise, for $\delta\ge n^{-0.1}$, we have proved that with probability at least $1-\exp(-2n^{0.1})$, the following holds. For any rule by which we explore $\hat D$ by revealing in-neighbours of previously revealed vertices, if we explore for at most $\delta n\log n$ steps, then there will be at most $\delta^2 n(\log n)^3$ vertices which are revealed more than once (as in-neighbours of different vertices).
\end{remark}
\subsection{Subcriticality}
In this subsection we prove \cref{lem:defective-bounding}(A).
\begin{proof}[Proof of \cref{lem:defective-bounding}(A)]
Let $(D,\vec{c})\sim\mb D_{t_{0}}(n,p)$. Recall that $|R(W,\vec Q)|$ is
the total number of vertices which would be assigned lists via the $(W,\vec Q)$-virtual
list-assignment process in the digraph $D$. Our goal is to prove that whp $|R(W,\vec Q)|\le2|W|(\log n)^{\ell+3}$ for all sets $W$ with at least $n^{0.9}$ vertices, and all tapes $\vec Q$.

Let $(\hat{D},\vec{c})\sim\hat{\mb D}_{t_{0}}(n,p)$, and for the rest
of the proof condition on an outcome of $T_{\hat{D},\vec{c}}$ satisfying \cref{D1,D2}. By \cref{lem:probably-simple,lem:marking-prelim}, it suffices to study $\hat D$ instead of $D$: we will show that with respect to virtual list-assignment processes on $\hat D$, with probability at least say $1-\exp(-n^{0.1})$, we have $|R(W,\vec Q)|\le 2|W| n(\log n)^{\ell+3}$
for all tapes $\vec Q$ and all sets $W$ with at least $n^{0.9}$ vertices.

\medskip\noindent\textit{Step 1: Covering with small vertex sets.}
We need to study $(W,\vec Q)$-virtual list-assignment processes for large vertex sets $W$. However, it is important that we only directly work with sets of vertices whose size is not too close to $n$ (so that we only ever have to explore a small part of $\hat D$, and our previously revealed vertices do not bias future revelations too much). It is easy to reduce our attention to such sets, by a covering argument. Note that for every set $W$ with at least $n^{0.9}$ vertices, we can find covering sets $W_1,\dots,W_b$ (with $W\subseteq W_1\cup \dots W_b$), such that each $W_i$ has size at most $n^{0.9}$, and the number of covering sets is $b\le \lfloor|W|/n^{0.9}\rfloor+1\le 2|W|/n^{0.9}$. Also, note that for any tape $\vec Q$, we can find tapes $\vec Q_1,\dots,\vec Q_b$ such that
\begin{equation}R(W,\vec Q)\subseteq \bigcup_{i=1}^b R(W_i,\vec Q_i).\label{eq:covering}\end{equation}
Indeed, in the $(W,\vec Q)$-virtual list-assignment process, every vertex $v$ which has a list ``owes'' its list to some $w\in W$, which started a periodic sequence of list-assignments and path-danger level assignments that led to $v$ being assigned a list. We should let $\vec Q_i$ be the sub-tape of $\vec Q$ corresponding to the list-assignments for the vertices which owe their list to the vertices in $W_i$. We remark that for \cref{eq:covering} it is very important that we have tapes recording the precise order in which list-assignments take place (due to some very subtle ways in which path-defective vertices can interact, it does not seem to be possible to fix a global rule describing the order in which vertices should be processed, such that \cref{eq:covering} always holds).

It now suffices to show that for any tape $\vec Q$ and any set $W$ of at most $n^{0.9}$ vertices, we have $|R(W,\vec Q)|\le n^{0.9}(\log n)^{\ell+3}$.

\medskip\noindent\textit{Step 2: Compressing the tape.}
For any $W$, we are interested only in ``maximal'' tapes $\vec Q$, which cannot be extended to continue the $(W,\vec Q)$-virtual list-assignment process any further. Such tapes actually contain a lot of redundant information: in order to specify an outcome of the $(W,\vec Q)$-virtual list-assignment process, for maximal $\vec Q$, we are really only concerned about:
\begin{itemize}
    \item for each $w\in W$, which in-neighbours of $w$ have already been assigned lists at the moment $w$ becomes virtual-defective, and
    \item the ``pathway'' via which each vertex gets assigned a list (specifically, if a vertex $u$ could be assigned a list via two different vertices $v,v'$ with $\on{pd}(v)\ne \on{pd}(v')$, then we need to know which of the two is actually responsible for $u$ receiving its list).
\end{itemize}

So, instead of considering tapes specifying the entire execution of a virtual list-assignment process, we consider ``guides'', which specify the relevant information much more efficiently. Crucially, this will permit us to take a union bound over all guides. To explain what information goes into a guide, we need some more discussion of how $\hat D$ is explored.

As in the proof of \cref{lem:defective-bounding}(B), we recall that $\hat{D}$ is determined by random matchings
$M_{\vec{\gamma}}$ (each in-stub $s\in S_{\vec{\gamma}}^{-}$ has
a partner $s\in S_{\vec{\gamma}}^{+}$, and vice versa). We can reveal these matchings gradually as we explore $\hat D$, revealing the identities of out-stubs according to random orderings $\prec_{\vec \gamma}$.
Recalling \cref{rem:defective-general} (and taking $\delta=n^{-0.08}$), note that with probability at least $1-\exp(-2n^{0.1})$, the orderings $\prec_{\vec \gamma}$ are such that if we run any virtual list-assignment process for at most $n^{0.91}$ steps, there are at most $n^{0.89}$ vertices revealed as in-neighbours of multiple different vertices (call these vertices ``duplicates''). Write $\mc E_{\mr{dup}}$ for the event that this property of the $\prec_{\vec \gamma}$ holds.

Now, for a vertex set $W$, a guide $G=(A,B)$ consists of:
\begin{itemize}
    \item A function $A:W\to  \mc N_{\log n}$, where $\mc N_{\log n}$ is the collection of all subsets of $\{1,\dots,\log n\}$. If $A(w)=I$, this indicates that the in-neighbours of $w$ indexed by $i\in I$ should be assigned lists before $w$ is made virtual-defective  (recall from \cref{D1} that every vertex has at most $\log n$ in-neighbours.).
    \item A set $B$ of at most $n^{0.89}\log n$ pairs $(u,i)\in V(\hat D)\times \{1,\dots,\log n\}$. If $(u,i)\in B$, this indicates that we should \emph{not} assign a list to the $i$-th neighbour of $u$ (or even reveal its identity) when processing $u$ (because we want this neighbour to be assigned a list via some other vertex).
\end{itemize}

Recall that we are interested in showing that virtual list-assignment processes terminate after at most $n^{0.9}(\log n)^{\ell+3}<n^{0.91}$ steps. So, we only need to include $(u,i)$ in $B$ if the $i$-th in-neighbour of $u$ is a duplicate. In practice, for every duplicate $v$, we include $(u,i_{u\to v})\in B$ for all but one of the out-neighbours $u$ of $v$ (where we write $i_{u\to v}$ for the index of $v$ among the out-neighbours of $u$). The pair $(u,i_{u\to v})$ that is \emph{not} included in $B$ indicates the pathway via which we wish $v$ to be assigned a list.

The upshot is that if $\mc E_{\mr{dup}}$ holds, then we will only ever need to consider sets $B$ of size $|B|\le n^{0.89}\log n$, in accordance with the definition of a guide (in addition to there being at most $n^{0.89}$ duplicates, we also recall from \cref{D1} that every vertex has at most $\log n$ out-neighbours).

For a set $W$ of at most $n^{0.9}$ vertices and a guide $G=(A,B)$, we define the \emph{$(W,G)$-guided} list-assignment process to be just like the list-assignment process defined in \cref{def:marking}, except that we skip steps (1), (2b) and (2c) (i.e., there are no colour-defective, cycle-defective or duplicate-defective vertices).
In addition, the guide $G$ is used in the following way:
\begin{itemize}
    \item In step (2d), we do not inspect the $i$-th in-neighbour of a vertex $u\in U$ (to see if we should assign it a list) if $(u,i)\in B$.
    \item In step (2d), when we are considering a vertex $u\in W$, and deciding which in-neighbours of $u\in W$ to inspect first, the in-neighbours indexed by $A(w)$ are always take priority.
    \item We add a new action (2e) (which takes priority over other actions): if there is a vertex $w\in W$ such that its in-neighbours indexed by $A(w)$ have already been assigned lists, then make $w$ defective. (We say $w$ is ``virtual-defective'').
\end{itemize}
Note that since we skip step (1), there are no vertices assigned lists at the start; the first list will always be assigned in step (2e).

Let $\hat R(W,G)$ be the set of vertices which would be assigned a list if we ran the $(W,G)$-guided list-assignment process. If $\mc E_{\mr{dup}}$ holds, and if $|\hat R(W,G)|\le n^{0.9}(\log n)^{\ell+3}$ for each set $W$ of at most $n^{0.9}$ vertices and each guide $G$, then it follows that $|R(W,\vec Q)|\le n^{0.9}(\log n)^{\ell+3}$ for each set $W$ of at most $n^{0.9}$ vertices and each tape $\vec Q$.

So, fix a set $W$ of at most $n^{0.9}$ vertices, and a guide $G$. Our goal will be to prove that with probability at least $1-\exp(-2n^{0.9}\log n)$ we have $|\hat R(W,G)|\le n^{0.9}(\log n)^{\ell+3}$. The desired result will then easily follow from the union bound over at most $n^{n^{0.9}}$ choices of $W$ and  at most \[(2^{\log n})^{n^{0.9}}\cdot (n^2)^{n^{0.89}\log n}\] choices of $G=(A,B)$.

\medskip\noindent\textit{Step 3: Iterative exposure of in-neighbourhoods.}
As previously mentioned, we gradually
reveal information about the $M_{\vec{\gamma}}$ on demand, as we explore $\hat{D}$ via the $(W,G)$-guided list-assignment process. Crucially,
at any moment where we have not yet explored very much of $\vec{D}$,
it is easy to see that the revelations at the next step are ``essentially
uniform''.

Specifically, suppose that so far we have only revealed the in-neighbours
of at most $n^{0.9}(\log n)^{\ell+3}$ vertices, and consider a vertex
$v$ with $\vec{c}(v)=\vec{\gamma}$ and $\deg^{-}(v)=d$, whose in-neighbours
have not yet been revealed. Given all the information revealed so far, let $\hat{\mathcal{N}}$ be the conditional
distribution of the set of outgoing stubs in $S_{\vec{\gamma}}^{+}$
which are matched with the $d$ incoming stubs corresponding to $v$,
and let $\mathcal{N}^*$ be a uniformly random set of $d$ stubs
in $S_{\vec{\gamma}}^{+}$ (sampled with replacement). We claim that
$\on{d}_{\mathrm{TV}}(\hat{\mathcal{N}},\mathcal{N}^*)=O(n^{-0.09})$.
To see this, note that by \cref{D1} we have only revealed the partners of
at most $n^{0.9}(\log n)^{\ell+4}$ stubs. By \cref{T-stub-count,lem:GW-probabilities-big} we have $|S_{\vec{\gamma}}^{+}|\ge n^{1-o(1)}$
and by \cref{D1} we have $d\le\log n$, so 
\[
\on{d}_{\mathrm{TV}}(\hat{\mathcal{N}},\mathcal{N}^*)\le O\left(\frac{n^{0.9}(\log n)^{\ell+4}+d}{|S_{\vec{\gamma}}^{+}|}\right)=O(n^{-0.09}),
\]
as claimed.

By the first part of \cref{T-critical}, the probability that a random stub
in $S_{\vec{\gamma}}^{+}$ corresponds to a vertex with in-degree $g$
is at most
\begin{equation}
\frac{\lambda^{g}e^{-\lambda}}{g!}+n^{-1/5}.\label{eq:outdegree-next}
\end{equation}
Moreover, if $v$ is not defective, then $L(v)$ contains the colours $\gamma_{t_0}$
and $\gamma'=\gamma_{t_0}+1$ (mod 3). By the second part of \cref{T-critical},
the probability that a random stub in $S_{\vec{\gamma}}^{+}$ is $\gamma'$-critical and
corresponds to a vertex with in-degree $g$ is at most 
\begin{equation}
\frac{0.99999}{\lambda}\cdot\frac{\lambda^{g}e^{-\lambda}}{g!}+n^{-1/5}.\label{eq:critical-next}
\end{equation}

To summarise \cref{eq:critical-next,eq:outdegree-next}, and the fact that $\on{d}_{\mathrm{TV}}(\hat{\mathcal{N}},\mathcal{N}^*)=O(n^{-0.09})$: up to some error terms, we can imagine that each of the $d$ in-neighbours of $v$ themselves have independent $\on{Poisson}(\lambda)$ in-degrees, and they are independently $\gamma'$-critical with probability at most $0.99999/\lambda$.

\medskip\noindent\textit{Step 4: Comparison with a branching process.} We now define an abstract branching process
which stochastically dominates the list-assignment process on $\hat D$.
 
Let $\mathcal{E}$ be the distribution of a random variable that is $\lfloor\log n\rfloor$ with probability
$n^{-0.08}$, and zero otherwise. For two probability measures $\mathcal{P},\mathcal{Q}$,
write $\mathcal{P}+\mathcal{Q}$ for the distribution of the sum of
independent random variables distributed as $\mathcal{P}$ and $\mathcal{Q}$.
Note that if $N\sim\on{Poisson}(\lambda)$ and $(X_{i})_{i=1}^{\infty}$
is a sequence of i.i.d. $\on{Bernoulli}(0.99999/\lambda)$ random variables, then $\sum_{i=1}^{N}X_{i}\sim\on{Poisson}(0.99999)$.
So, given the considerations in the previous section, $|\hat R(W,G)|$ is stochastically dominated by the total population
in the following non-homogeneous branching process.
\begin{enumerate}
\item There are $|W|$ roots (``generation zero''), whose numbers of offspring are given by $\deg^{-}(v)$
for $v\in W$.
\item In every generation divisible by $\ell+1$ (apart from generation zero), the offspring distribution
is $\on{Poisson}(\lambda)+\mathcal{E}$.
\item In every generation not divisible by $\ell+1$, the offspring distribution
is $\on{Poisson}(0.99999)+\mathcal{E}$.
\item If the total population ever reaches $n^{0.9}(\log n)^{\ell+3}$,
terminate the process and artificially add $n$ offspring to some
vertex.
\end{enumerate}

\medskip\noindent\textit{Step 5: ``Contracting'' the branching process.} In order
to apply off-the-shelf concentration inequalities from the literature, we compare
the above branching process to a (homogeneous) Galton-Watson process (essentially,
we ``contract'' the process into blocks of $\ell+1$ generations,
each of which can be viewed as a single generation in a Galton--Watson
process). Let $q=\sum_{v\in W}\deg^{-}(v)\le n^{0.9}\log n$ be the total number of individuals at generation 1 (i.e., the number of offspring of the $|W|$ roots)

Let $\mathcal{R}$ be the distribution of the number of individuals at generation $\ell+1$, in a branching process with just one root individual (at generation zero), where generations zero through $\ell-1$ have offspring distribution $\on{Poisson}(0.999999)+\mathcal{E}$, and generation $\ell$ has offspring distribution $\on{Poisson}(\lambda)+\mathcal{E}$.
Then, (except for the artificial termination in (4)), the ``contracted'' branching
process described above corresponds to $q$ independent Galton--Watson
processes with offspring distribution $\mathcal{R}$. Let $Y_{1},\dots,Y_{q}$
be the total populations of $q$  such independent processes.

Note that our ``contraction'' operation reduces the total population by at most a factor of $(\log n)^{\ell}$. Let $Y^{*}=(\log n)^{\ell}(Y_{1}+\dots+Y_{q})$,
so that $|\hat R(W,G)|$ is stochastically dominated by
\[Y^*+n\one\{Y^{*}\ge n^{0.9}(\log n)^{\ell+3}\}.\]

It now suffices to show that
$\Pr[Y^{*}\ge n^{0.9}(\log n)^{\ell+3}]\le \exp(-2n^{0.9}\log n)$. This will follow from a standard Chernoff bound for Galton--Watson processes, after computing some relevant quantities.

\medskip\noindent\textit{Step 6: Computations.} Recall that the cumulant generating function (cgf) $\kappa_X$ of a random variable $X$ is given by $\theta\mapsto \log(\mb E\exp(\theta X))$. The cgf of $\on{Poisson}(\alpha)$ is $\kappa_\alpha:\theta\mapsto \alpha(e^\theta-1)$, and the cgf of $\mc E$ is
\[
\kappa_{\mc E}:\theta\mapsto \log\left(\frac{\exp(\theta \lfloor\log n\rfloor)-1}{n^{0.08}}+1\right).\]
The cgf of $\on{Poisson}(\alpha)+\mathcal{E}$ is then $\kappa_\alpha^*:=\kappa_\alpha+\kappa_{\mc E}$. Iterating the law of total expectation, we can see that the cgf of $\mc R$ is
\[\kappa_{\mc R}:z\mapsto\kappa_{\lambda}^{*}(\kappa_{\beta}^{*}(\kappa_{\beta}^{*}(\dots\kappa_{\beta}^{*}(z)\dots)))\]
(where $\beta=0.99999$ and $\kappa_{\beta}^{*}$ is iterated $\ell$ times).

Now, note that $\kappa_{\alpha}'(0)=\alpha$ so by continuity there is some $\theta_\alpha>0$ such that $\kappa_{\alpha}(\theta)\le1.000001 \alpha \theta$ for $0\le \theta\le\theta_\alpha$. Also, note that if $0\le \theta\le 0.08$ then $\kappa_{\mc E}(\theta)=o(1)$. So, with $\theta^*=\min(\theta_{1/2},\theta_C,0.08)$, we have
\[\kappa_{\mc R}(e^{\theta^*})\le(1.000001)^{\ell+1}C(0.99999)^{\ell}\theta^*+o(1)\le \theta^*/2\]
(recall that $\ell=10^{10}C$ and that $\lambda\le C$).

Let $h:x\mapsto\sup_{\theta\ge0}(\theta x-\kappa_{\mc R}(x))$ be the Legendre transform of $\kappa_{\mc R}$, so the above considerations show that $h(1)\ge\theta^*/2>0$ (note that this does not depend on $n$).
By a Chernoff bound for subcritical Galton--Watson processes (see
for example \cite[Lemma~1.9]{DM10}), for each $i$ we have
\[
\Pr[Y_{i}\ge k]\le\exp(-h(1)k).
\]

Now, for $k=n^{0.9}(\log n)^{3}$ we have
\begin{align*}
\Pr[Y^{*}\ge n^{0.9}{n}(\log n)^{\ell+3}]=\Pr[Y_{1}+\dots+Y_{q}\ge k]&\le \sum_{\substack{k_{1},\dots,k_{q}\in\mb N\\
k_{1}+\dots+k_{q}=k
}
}\prod_{i=1}^{q}\Pr[Y_{i}\ge k_{i}]\\
&\le n^{q}\exp(-h(1)k)\le \exp(-2n^{0.9}\log n),
\end{align*}
as desired.
\end{proof}
\subsection{Putting everything together}

We are finally ready to prove \cref{lem:marking}.
\begin{proof}[Proof of \cref{lem:marking}]
Let $t_{0}=(\log \log n)^2$ and $(D,\vec{c})\sim\mb D_{t_{0}}(n,p)$.
Then, consider the list-assignment process described in \cref{def:marking}, on $D$, with initial
colouring $c_{t_{0}}$ and with $\ell=10^{10}C$. Let $\tau$ be the total number of steps that this process takes, let $U(i)$ be the set of vertices which have been assigned lists after $i$ steps of the list-assignment process, and let $W_{\mr{dup}}(i)$ be the set of vertices that have been marked as duplicate-defective after $i$ steps of the list-assignment process.

Our objective is to prove that whp $|U(\tau)|=o(n)$; our final list-assignment will then satisfy the conditions in \cref{lem:marking} (with $c$ being the restriction of $c_{t_{0}}$ to the vertices not in $U$).

Let $\delta=1/(\log n)^{\ell+10}$, and let $\tau_\delta=\min(i:|U(i)|\ge \delta n)$ be the first time that $\delta n$ vertices have been assigned lists by our list-assignment process. Note that if $\tau_\delta=\infty$ then $|U(\tau)|=o(n)$ and we are done. By \cref{lem:defective-bounding}(B), whp either $\tau_\delta=\infty$ or
\begin{enumerate}
\item $W_{\mr{dup}}(\tau_\delta)\le \delta^2 n (\log n)^3=\delta n/(\log n)^{\ell+7}$.
\end{enumerate}
Also, by \cref{D3} and \cref{T1}, and \cref{lem:defective-bounding}(A), whp:
\begin{enumerate}
\setcounter{enumi}{1}
\item at every point in the list-assignment process, there
are at most $\sqrt n\le \delta n/(\log n)^{\ell+7}$ cycle-defective vertices, and
\item there are at most $0.9999^{t_0}n+n^{3/4}\le \delta n/(\log n)^{\ell+7}$ colour-defective vertices
(which are marked as such at the start of the list-assignment process),
and
\item for every set $W$ of at least $n^{0.9}$
vertices, and every tape $\vec Q$, we have $|R(W,\vec Q)|\le 2|W|(\log n)^{\ell+3}$
\end{enumerate}
But note that (1)--(4) cannot simultaneously hold. Indeed, suppose for the purpose of contradiction that (1)--(4) all hold. Let $W$ be the set of vertices which are colour-defective, cycle-defective or duplicate-defective at time $\tau_\delta$. By (1)--(3)  we have $|W|\le 3\delta n/(\log n)^{\ell+7}$.
So, by
(4), for some suitable tape $\vec Q$ we have
\[|U(\tau_\delta)|\le|R(W,\vec Q)|\le 2\left(\frac{3\delta n}{(\log n)^{\ell+7}}\right)(\log n)^{\ell+3}< \delta n,\]
which contradicts the definition of $\tau_\delta$.
\end{proof}

\section{Majority list-colouring given an acyclic partition}\label{sec:acyclic-partition}

\begin{proof}[Proof of \cref{lem:acyclic-partition}]
For each possible list $L$, let $U_L$ be the set of vertices with list $L$. We can linearly order the vertices of $U_L$ in such a way that all arcs induced by $U_L$ go ``backwards'' in the ordering (i.e., if $v\to u$ is an arc in $D[U_L]$, then $u\prec_L v$ according to our ordering $\prec_L$ on $U_L$). Then, independently for each $L$, we can greedily choose colours $c(v)\in L$ for each $v\in U_L$ (in the order specified by $\prec_L$) as follows:
\begin{itemize}
    \item Recall that all the vertices that are not in $U$ already come with a specified colour.
    \item When it comes time to colour vertex $v$, we have already chosen colours for all out-neighbours $u\in U_L$ of $v$ (by the choice of the ordering $\prec_L$).
    \item For each possible list $L'\ne L$, and each out-neighbour $u\in U_{L'}$ of $v$, imagine that $c(u)$ is coloured with the unique colour in $L\cap L'$ (this is the ``most pessimistic'' assumption).
    \item There is at most one colour appearing on more than half the out-neighbours of $v$ (according to the real and imagined colour choices). So, we can make a choice $c(v)\in L$ which is not this colour.\qedhere
\end{itemize}
\end{proof}

\section{Finishing the proof of \texorpdfstring{\cref{thm:majority-Gnp}}{Theorem~\ref{thm:majority-Gnp}}}\label{sec:finish}
Finally we can complete the proof of \cref{thm:majority-Gnp}. First, it is straightforward to take care of the regime $p< 0.1/n$.
\begin{proof}[Proof of \cref{thm:majority-Gnp} in the case $p<0.1/n$]
Let $p'=1-(1-p)^2\le 0.2/n$. Note that if we remove the directions on the edges of $D\sim \mb D(n,p)$ (antiparallel arcs become a single undirected edge), then we obtain a random graph $G\sim \mb G(n,p')$. It is well-known (see for example \cite[Corollary~5.8]{BolBook}), that such sparse random graphs whp have at most one cycle in every component, and are therefore (properly) 3-colourable. Note that a proper 3-colouring is of $G$ is trivially a majority 3-colouring of $D$.
\end{proof}
Next, we use \cref{lem:marking,lem:acyclic-partition,lem:2degenerate} to handle the case where $p$ has order of magnitude $1/n$.
\begin{proof}[Proof of \cref{thm:majority-Gnp} in the case $0.1/n\le p\le O(1/n)$] Let $D\sim\mb D(n,p)$ for any $0.1/n\le p=O(1/n)$. First, we recall the conclusion of \cref{lem:marking}: whp we can find a subset $U\subseteq V(D)$, an assignment of a colour $c(v)\in \{1,2,3\}$ to each $v\notin U$, and an assignment of a list $L(v)$ to each $v\in U$, such that the following hold.
\begin{enumerate}[{\bfseries{L\arabic{enumi}}}]
    \item\label{L1'} For any completion of our partial colouring $c$, obtained by assigning colours $c(v)\in L(v)$ to the vertices $v\in U$, all $v\notin U$ are majority-coloured with respect to $c$.
    \item\label{L2'} $|U|=o(n)$.
    \item\label{L3'} Each $v\in U$ has list size $|L(v)|=2$ or $|L(v)|=3$.
    \item\label{L4'} Every directed cycle in $D[U]$ contains at least two vertices with list size 3.
    \item\label{L5'} There is $\ell=O(1)$ such that every length-$\ell$ directed path in $D[U]$ has a vertex with list size 3.
\end{enumerate}

Combining \cref{L2'} with \cref{lem:2degenerate}, whp we have the following additional property.
\begin{enumerate}[{\bfseries{L\arabic{enumi}}}]
\setcounter{enumi}{5}
    \item\label{L6'} Every subset $S\subset U$ spans at most $(1+0.1/\ell)|S|$ arcs.
\end{enumerate}
It now suffices (by \cref{L1'}) to prove that if properties \cref{L2',L3',L4',L5',L6'} hold, then there is an assignment of colours $c(v)\in L(v)$ to each $v\in U$, such that every vertex $v\in U$ is majority-coloured. We will prove this via \cref{lem:acyclic-partition} (recalling \cref{L3}, we need to delete a colour from each of the lists of size 3, in such a way that the assumption of \cref{lem:acyclic-partition} holds).

Let $W\subseteq U$ be the set of vertices with list size 3, and let $G$ be the (undirected) graph with vertex set $W$ obtained by putting an edge $ww'$ whenever there is a directed path between $w$ and $w'$ all of whose internal vertices are in $U\setminus W$ (by \cref{L5'}, such a path has length at most $\ell+1$, with at most $\ell$ internal vertices). In particular, we put an edge $ww'$ whenever there is an edge between $w$ and $w'$ in either direction.

We next claim that $G$ has a proper 3-colouring. Indeed, for any subset $S\subseteq V(G)$, if we consider the set $S'\subseteq V(D)$ obtained by adding to $S$ all vertices of $U$ involved in all the paths of $D$ which define the edges of $G[S]$, then $e(D[S'])\ge e(G[S])+|S'\setminus S|$, while $|S'\setminus S|\le \ell \cdot e(G[S'])$ by \cref{L5'}, as per the discussion in the previous paragraph (here we write $e(G)$ for the number of edges or arcs in a graph or digraph $G$).
On the other hand, $e(D[S'])\le (1+0.1/\ell)|S'|$ by \cref{L6'}. So, we have
\[e(G[S])\le e(D[S'])-|S'\setminus S|\le (1+0.1/\ell)|S'|-|S'\setminus S|=|S|+(0.1/\ell)|S'\setminus S|\le 1.1 |S|,\]
meaning that $G[S]$ has average degree at most $2.2$, and therefore has a vertex with degree at most 2. Since this is true for all $S\subseteq V(G)$, there is a proper 3-colouring $c_G:W\to \{1,2,3\}$ of $G$.

Delete the colour $c_G(v)$ from the list $L(v)$, for each defective $v\in W$. After this deletion, each $v\in U$ has a list of size 2; to apply \cref{lem:acyclic-partition} it suffices to prove that for each of the three possible lists $L$, there is no directed cycle among the vertices which have that list. Indeed, \cref{L4'} implies that any cycle in $U$ must contain at least two vertices in $W$, and if we take such a pair at minimal distance along the cycle, then $ww'$ must comprise an edge in $G$. But then $c_G(w)\ne c_G(w')$, so $w$ and $w'$ have different lists.
\end{proof}

It remains to consider the case $p=\omega(1/n)$. Note that the case $p\ge 200\log n/n$ is completely trivial, because then whp all vertices have out-degree at least  $100\log n$, and a uniformly random 3-colouring is a majority colouring whp (as observed in \cite[Theorem~3]{KOSZW17}). So, we only really need to worry about the range where $\omega(1/n)\le p\le  200\log n/n$. We use a much cruder version of the above proof for the case $p= \Theta(1/n)$ (we consider a uniformly random colouring, assign lists based on this colouring, and then apply \cref{lem:acyclic-partition}). We will need the following quantitative variant of \cref{lem:2degenerate}.

\begin{lemma}\label{lem:1.5density}
Let $p=\omega(1/n)$ and $G\sim\mb G(n,p)$. Then whp every vertex subset $S$ with $|S|\le (np)^{-21} n$ spans at most $1.1|S|$ edges.
\end{lemma}

\begin{proof}
Let $\delta=(np)^{-21}$. The probability that there exists some $S$
violating the lemma statement is at most 
\begin{align*}
\sum_{s=1}^{\delta n}\binom{n}{s}\binom{s^{2}}{1.1s}p^{1.1s} & \le\sum_{s=1}^{\delta n}\left(\frac{en}{s}\right)^{s}\left(\frac{es^{2}}{1.1s}\right)^{1.1s}p^{1.1s}=\sum_{s=1}^{\delta n}\left(\frac{e^{2.1}}{1.1^{1.1}}np\right)^{s}(sp)^{0.1s}\\
 & =\sum_{s=1}^{\delta n}\big(O(1)\,np\,(sp)^{0.1}\big)^{s}\le\sum_{s=1}^{\delta n}\left(\frac{O(1)}{np}\right)^{s}=\frac{O(1)}{np}=o(1),
\end{align*}
as desired. (Here we used that when $s\le\delta n$ we have $(sp)^{0.1}\le(np)^{-2}$).
\end{proof}

Now we are finally ready to prove the remaining cases of \cref{thm:majority-Gnp}.

\begin{proof}[Proof of \cref{thm:majority-Gnp} in the case $\omega(1/n)\le p\leq 200\log n/n$]
Consider a uniformly random 3-colouring, and say that a vertex $v$ is \emph{robustly majority-coloured} if at most $\deg^{+}(v)/2-1$ of its out-neighbours have the same colour as $v$. Let $U_0$ be the set of vertices which are not robustly majority-coloured.

For each $m,q$, let $B(m,q)\sim \on{Binomial}(m,q)$. Then, for each vertex $v$, we have
\begin{align*}
    \Pr[v\in U_0]&\le \Pr[\deg^+(v)\le np/2]+\sup_{d\ge np/2}\Pr[v\in U_0\,|\,\deg^+(v)=d]\\
    &=\Pr[B(n-1,p)\le np/2]+\sup_{d\ge np/2}\Pr[B(d,1/3)\ge d/2-1]=e^{-\Omega(np)},
\end{align*}
by a Chernoff bound. By linearity of expectation and Markov's inequality (recalling that $np=\omega(1)$), whp the number of vertices $|U_0|$ that fail to be robustly majority-coloured is at most $e^{-\Omega(np)}n$. It suffices to show that this property, together with the property in \cref{lem:1.5density}, implies the statement of \cref{thm:majority-Gnp} (so, for the rest of the proof, we no longer use the randomness of $D$ or our random 3-colouring).

Starting with $U=U_0$, we iteratively expand the set $U$ as follows. Whenever there is a vertex outside $U$ with more than one out-neighbour in $U$, add that vertex to $U$. We claim that this process cannot continue for more than $2|U_0|$ steps; indeed, after $2|U_0|$ steps we would have $|U|=3|U_0|$ and $e(D[U])\ge 2(2|U_0|)$, which would contradict the property in \cref{lem:1.5density}.

We have now found a set $U$ of only $e^{-\Omega(np)}n$ vertices such that all vertices outside $U$ are robustly majority-coloured and have at most one out-neighbour in $U$ (this means that the vertices outside $U$ will remain majority-coloured no matter how we recolour the vertices in $U$). It now suffices to colour the vertices in $U$. To this end, note that the property in \cref{lem:1.5density} implies that the graph underlying $D[U]$ is 2-degenerate, so has chromatic number at most 3. So, we can partition $D[U]$ into three independent sets, assign to each of these independent sets a list of size 2, and apply \cref{lem:acyclic-partition}. (It would also be easy to find an appropriate colouring with a direct greedy argument).
\end{proof}

\section{Non-constructive majority 2-colouring}\label{sec:majority-2-colouring}

In this section we prove \cref{thm:majority-2-colouring}.
The key ingredient for the
proof of \cref{thm:majority-2-colouring} is the following lemma estimating the first and second
moments of the number of majority bisections.
\begin{lemma}\label{lem:moments}
Let $D\sim\mb D(n,p_{n})$, where $n p_{n}(1-p_n)\to\infty$, and let $X$ be
the number of majority bisections in $D$. Then
\begin{enumerate}
\item $\mb E X\ge e^{-o(n)}$,
\item $\mb E X^{2}\le e^{o(n)}$.
\end{enumerate}
\end{lemma}

The proof of \cref{lem:moments} is essentially a large-deviations calculation, similar
to the calculations in \cite{DGZ}. We defer this proof until \cref{subsec:moments}.

By the Paley-Zygmund inequality, it follows from \cref{lem:moments} that $\Pr[X\ne0]\ge(\mb E X)^{2}/\mb E X^{2}=e^{-o(n)}$.
That is to say, it is not exponentially unlikely that $D$ has a majority
bisection. In order to deduce from this that $D$ has an almost-majority
2-colouring whp, we adapt a concentration trick that seems to have been first used by Frieze~\cite{Fri90}, in the same
way as \cite{DGZ}. Namely, we define a second random variable $Z$ measuring
(in some appropriate sense) how close to a majority 2-colouring we
can obtain, observe that $Z$ is concentrated around its mean, and
deduce that $\mb E Z=o(n)$ (otherwise it would not be possible to have
$\Pr[Z\ne0]\ge e^{-o(n)}$). We will take our random variable $Z$
to be the minimum \emph{defect} of our random colouring, defined as
follows.
\begin{definition}
Given a 2-colouring $c:V(G)\to\{1,2\}$ and a vertex $v\in V(G)$, we define the \emph{defect
}
\[
\operatorname{def}(v;c)=\min\big(|\{w\in N^{+}(v):c(w)=c(v)\}|-|\{w\in N^{+}(v):c(w)\ne c(v)\}|,\;0\big).
\]
In words, the defect is zero if $v$ is majority-coloured, and otherwise
the defect is the number of same-coloured out-neighbours of $v$ minus
the number of oppositely-coloured out-neighbours of $v$. Then, the
defect of the entire colouring $c$ is defined as 
\[
\operatorname{def}(c)=\sum_{v\in V(G)}\operatorname{def}(v;c).
\]
\end{definition}

We need the fact that the minimum defect of a random digraph is tightly
concentrated, as follows.
\begin{lemma}\label{lem:defect-concentration}
Let $D\sim\mb D(n,d/n)$ with $d^*:=\min(d,n-d)\ge1$, and let $Z$ be the
minimum defect among all bisections of $D$. Then for large $n$ and any $\varepsilon\ge0$
we have 
\[
\Pr[|Z-\mb E Z|\ge\varepsilon\sqrt{d^*}n]\le\exp(-\varepsilon^{2}n/20).
\]
\end{lemma}

\begin{proof}
If any edge is added or removed from $D$, then $Z$ changes by at
most 2. So, a bounded-difference inequality such as \cite[Theorem~2.11]{Kwa20} shows
that 
\[
\Pr[|Z-\mb E Z|\ge t]\le\exp\left(-\frac{t^{2}}{16n(n-1)\min(p,1-p)+4t}\right)
\]
for any $t\ge0$. The desired result follows.
\end{proof}
We also need the fact that in a random graph, there is no bisection in which many vertices have small positive defect.
\begin{lemma}\label{lem:few-low-defect}
Consider $d_n$ and $\varepsilon_n$ such that $\varepsilon_n\to0$ and $d^*_n:=\min(d_n,n-d_n)\to\infty$. Then whp $D\sim\mb D(n,d_n/n)$ has the property that for every bisection $c:V(D)\to\{1,2\}$,
there are at most $o(n)$ vertices $v$ with $0<\operatorname{def}(v;c)\le \varepsilon_n \sqrt{d^*_n}$.
\end{lemma}

\begin{proof}
Fix a bisection $c:V(G)\to\{1,2\}$. For $i\in\{1,2\}$, let $\deg_{i}(v)$
be the number of out-neighbours of $v$ with $c(v)=i$. Note that we can
only have $0<\operatorname{def}(v;c)\le o(\sqrt{d^*_n})$ if $|\deg_{1}(v)-\deg_{2}(v)|\le o(\sqrt{d^*_n})$.
We will show that with probability $1-o(2^{-n})$, at most $o(n)$ vertices
have $|\deg_{1}(v)-\deg_{2}(v)|\le o(\sqrt{d^*_n})$; the desired result
will follow from a union bound over bisections.

To this end, note that all $2n$ random variables of the form $\deg_{i}(v)$
are independent. Note that $\deg_{2}(v)$ has a $\operatorname{Binomial}(n/2,d_n/n)$
distribution\footnote{Strictly speaking, if $n$ is odd then the distribution is $\operatorname{Binomial}(\lfloor n/2\rfloor,d_n/n)$ or $\operatorname{Binomial}(\lceil n/2\rceil,d_n/n)$ (i.e., there are some rounding considerations). This does not materially affect the rest of the proof.},
so by direct calculation or an anticoncentration inequality
such as \cite[Lemma~8.1]{CV08}, we have $\max_{x\in\mb N}\Pr[\deg_{2}(v)=x]=O(1/\sqrt{d^*_n})$.
It follows that
\begin{align*}
\Pr[|\deg_{1}(v)-\deg_{2}(v)|\le\varepsilon\sqrt{d^*_n}]&=\mb E\left[\Pr\left[\deg_{2}(v)\in[\deg_{1}(v)-\varepsilon\sqrt{d^*_n},\deg_{1}(v)+\varepsilon\sqrt{d^*_n}]\,\middle|\,\deg_{1}(v)\right]\right]\\
&=O(\varepsilon_n)\le\sqrt{\varepsilon_n}.
\end{align*}
So, with $\delta=10/\log(1/\varepsilon_n)=o(1)$, we have
\begin{align*}
\Pr[|\deg_{1}(v)-\deg_{2}(v)|\le\varepsilon_n\sqrt{d^*_n}\text{ for }\delta n\text{ different }v]&\le\binom n{\delta n}(\sqrt \varepsilon_n)^{\delta n}\\
&\le \exp\Big(\big(\delta \log(e/\delta)-(1/2)\delta\log(1/\varepsilon_n)\big)n\Big)=o(2^{-n}).
\end{align*}
The desired result follows.
\end{proof}
\begin{proof}
[Proof of \cref{thm:majority-2-colouring}]
Let $D\sim\mb D(n,p_n)$ with $np_n(1-p_n)\to \infty$, and let $d^*_n=\min(np_n,n-n p_n)$ (so $d^*_n\to\infty$). We wish to prove that $D$ has an $o(1)$-almost-majority bisection whp. 

Let $Z$ be the minimum defect among all bisections
of $D$. Using the Paley-Zygmund inequality we have 
\[
\Pr[Z=0]=\Pr[X\ne0]\ge\frac{(\mb E X)^{2}}{\mb E X^{2}}\ge e^{-o(n)}.
\]
On the other hand we have 
\[
\Pr[Z=0]\le\Pr[|Z-\mb E Z|\ge\mb E Z].
\]
Combining these inequalities with \cref{lem:defect-concentration}, we see that $\mb E Z=o(\sqrt{d^*_n}n)$,
so using \cref{lem:defect-concentration} again, we see that whp $Z=o(\sqrt{d^*_n}n)$. That is to say,
whp there is a bisection with defect $o(\sqrt{d^*_n}n)$. By \cref{lem:few-low-defect}, whp there
are at most $o(n)$ vertices responsible for this defect, as desired.
\end{proof}

\subsection{Moment calculations}\label{subsec:moments}
In this subsection we prove \cref{lem:moments}. First, we need to understand the probability that a particular vertex is majority-coloured, or that a pair of vertices are both majority-coloured. In a directed graph $D$ we define the \emph{overlap} between two colourings $c,c':V(D)\to\{1,2\}$
 to be the fraction of vertices $v\in V(D)$ for which $c(v)=c'(v)$. 
\begin{lemma}\label{lem:moments-input}
Fix a constant $\varepsilon>0$. Let $D\sim\mb D(n,p_{n})$ with $np_n(1-p_n)\to\infty$, and
fix a pair of bisections $c,c':V(D)\to\{1,2\}$ with overlap $\alpha \in[\varepsilon,1-\varepsilon]$.
For a vertex $v$, let $\mathcal{E}_{v}$ and $\mathcal{E}_{v}'$
be the events that $v$ is majority-coloured with respect to $c$
and with respect to $c'$. Then, for every vertex $v$:
\begin{enumerate}
\item $\displaystyle \Pr[\mathcal{E}_{v}]=1/2+o(1)$;

\item ${\displaystyle \Pr[\mathcal{E}_{v}\cap\mathcal{E}_{v}']=\frac{1}{\pi}\arctan\left(\sqrt{\frac{\alpha}{1-\alpha}}\right) + o(1)}$ if $c(v)=c'(v)$;
\item \vspace{5pt}${\displaystyle \Pr[\mathcal{E}_{v}\cap\mathcal{E}_{v}']=\frac{1}{\pi}\arctan\left(\sqrt{\frac{1-\alpha}{\alpha}}\right) + o(1)}$ if $c(v)\ne c'(v)$.
\end{enumerate}

\end{lemma}

\begin{proof}
Fix a vertex $v$. First, (1) is easy to prove by a direct computation,
but as a warm-up for (2) and (3) we give a proof using Gaussian approximation.
For $i\in\{1,2\}$, let $\deg_{i}(v)$ be the number of neighbours
of $v$ with $c(v)=i$. Let 
\[
X=\frac{\deg_{1}(v)-\deg_{2}(v)}{\sqrt{np_n(1-p_n)}},
\]
and consider a Gaussian random variable $Z\sim\mathcal{N}(0,1)$.
By the Berry--Esseen theorem \cite{Ber41,Ess42}, for any interval $I\subseteq\mb R$
we have 
\[
\Pr[X\in I]=\Pr[Z\in I]+O\left(\frac{1}{\sqrt{np_{n}(1-p_{n})}}\right).
\]
Note that $\mathcal{E}_{v}$ is precisely the event that $X\le0$
(if $c(v)=1$) or that $X\ge0$ (if $c(v)=2$). So, (1) follows from the approximation of $X$ by $Z$ and
the symmetry of $\mathcal{N}(0,1)$.

We next prove (2) in the case that $c(v)=c'(v)=1$ (then (3) and the other case of (2) follow from a very similar calculation). For $i,j\in\{1,2\}$, let $V_{ij}$ be the
set of vertices $w$ with $(c(w),c'(w))=(i,j)$, and note that 
\[
|V_{ij}|=\begin{cases}
\alpha n/2 & \text{if }i=j,\\
(1-\alpha)n/2 & \text{if }i\ne j.
\end{cases}
\]
Let $\deg_{ij}(v)=N(v)\cap V_{ij}$ be the number of neighbours that
$v$ has in $V_{ij}$, so $\mathcal{E}_{v}\cap\mathcal{E}_{v}'$ is
the event that 
\[
\deg_{11}(v)+\deg_{12}(v)\le\deg_{21}(v)+\deg_{22}(v)\text{ and }\deg_{11}(v)+\deg_{21}(v)\le\deg_{12}(v)+\deg_{22}(v),
\]
or equivalently that
\[
\deg_{22}(v)-\deg_{11}(v)\ge|\deg_{21}(v)-\deg_{12}(v)|.
\]

Let
\[
\vec{X}=\left(\frac{\deg_{22}(v)-\deg_{11}(v)}{\sqrt{\alpha np_{n}(1-p_{n})}},\;\frac{\deg_{21}(v)-\deg_{12}(v)}{\sqrt{(1-\alpha)np_{n}(1-p_{n})}}\right)\in\mb R^{2},
\]
and consider a bivariate standard Gaussian random vector $\vec{Z}\in\mb R^{2}$.
By a multivariate
Berry--Esseen theorem (see for example \cite{Rai19}), for any convex $U\subseteq\mb R^{2}$
we have 
\[
\Pr[\vec{X}\in U]=\Pr[\vec{Z}\in U]+O\left(\frac{1}{\sqrt{\alpha(1-\alpha)np_{n}(1-p_{n})}}\right).
\]
In particular
\begin{align*}
\Pr[\mathcal{E}_{v}\cap\mathcal{E}_{v}'] & =\Pr[\sqrt{\alpha}Z_{1}\ge\sqrt{1-\alpha}|Z_{2}|]+o(1).
\end{align*}

Now, for $\vec{x}\in\mb R^{2}$, let $\angle\vec{x}\in(-\pi,\pi]$ be
the angle of $\vec{x}$ when expressed in polar coordinates. By the
rotational invariance of $\vec{Z}$, we have
\begin{align*}
\Pr[\sqrt{\alpha}Z_{1}\ge\sqrt{1-\alpha}|Z_{2}|] & =\Pr\left[|\angle\vec{Z}|\le\arctan\left(\sqrt{\frac{\alpha}{1-\alpha}}\right)\right]\\
 & =\frac{2\arctan\left(\sqrt{\frac{\alpha}{1-\alpha}}\right)}{2\pi}.
\end{align*}
The desired conclusion follows.
\end{proof}
Next, we need a basic numerical inequality to understand the contribution
to $\mb E X^{2}$ from the various overlaps $\alpha$.
\begin{lemma}\label{lem:calculus-inequality}
Define the function $f:(0,1)\to\mb R$ by
\[
f(\alpha)=\alpha\left(\log\left(\arctan\left(\sqrt{\frac{\alpha}{1-\alpha}}\right)\right)-\log\alpha\right)+(1-\alpha)\left(\log\left(\arctan\left(\sqrt{\frac{1-\alpha}{\alpha}}\right)\right)-\log(1-\alpha)\right)+\log\left(\frac2\pi\right).
\]
Then $f(\alpha)\le0$ for all $\alpha\in(0,1)$, and $f(\alpha)=0$
if and only if $\alpha = 1/2$.
\end{lemma}

\begin{figure}
\begin{centering}
\begin{tikzpicture}[]
\begin{axis}[
width=12cm,height=4cm,
scaled ticks=false,
yticklabel style={/pgf/number format/fixed},
every axis plot/.append style={thick},
axis x line*=bottom,
axis y line*=left,
xtick={0,0.2,0.4,0.6,0.8,1},
ytick={0,-0.02,-0.04},
ymin=-0.055,
]
\addplot [blue,domain=0:1,samples=500]
{
x*ln(atan(sqrt(x/(1-x)))/180*pi)-x*ln(x)+(1-x)*ln(atan(sqrt((1-x)/x))/180*pi)-(1-x)*ln(1-x)+ln(2/pi)
};
\end{axis}
\end{tikzpicture}
\par\end{centering}
\caption{\label{fig:f}A plot of the function $f$ from \cref{lem:calculus-inequality}.}
\end{figure}
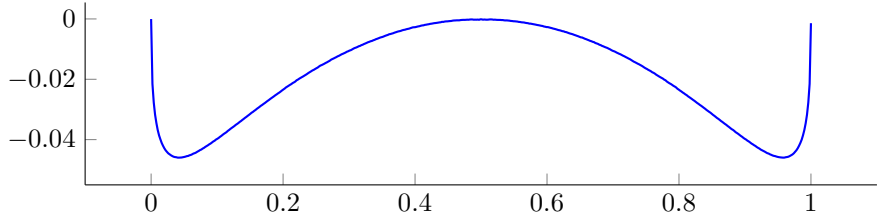

We were not able to find a clean proof of \cref{lem:calculus-inequality}, though it is very believable given a plot of $f$ (see \cref{fig:f}). In 
\cref{sec:calculus-inequality}
we sketch how to formally verify it by combining some computer calculations with Taylor's theorem.
\begin{proof}
[Proof of \cref{lem:moments}]
For a bisection $c:V(D)\to\{1,2\}$, let $\mathcal{E}_{v}^{c}$
be the event that $v$ is majority-coloured with respect to $c$.
If we fix a particular $c$, then the $n$ events $\mathcal{E}_{v}^{c}$
are independent, so the probability that $c$ is a majority-colouring
is $(1/2+o(1))^{n}$ by \cref{lem:moments-input}(1). So,
\[
\mb E X=\binom{n}{n/2}(1/2+o(1))^{n}\ge e^{-o(n)},
\]
proving \cref{lem:moments}(1). For \cref{lem:moments}(2), let 
\begin{align*}
p_{\alpha}(n) & =\left(\frac{1}{\pi}\arctan\left(\sqrt{\frac{\alpha}{1-\alpha}}\right)\right)^{\alpha n}\left(\frac{1}{\pi}\arctan\left(\sqrt{\frac{1-\alpha}{\alpha}}\right)\right)^{(1-\alpha)n}.
\end{align*}

Now, the estimates in \cref{lem:moments-input} hold for any (arbitrarily small) constant $\varepsilon>0$, so they must also hold when $\varepsilon\to 0$ sufficiently slowly. (Concretely, inspecting the proof of \cref{lem:moments-input}, we can take $\varepsilon = 1/\sqrt{\min(np_{n},n-np_n)}=o(1)$). If we consider particular $c,c'$ with overlap $\alpha\in[\varepsilon,1-\varepsilon]$, then the $n$
events of the form $\mathcal{E}_{v}^{c}\cap\mathcal{E}_{v}^{c'}$
are independent, so the probability that $c,c'$ are both majority-colourings
is 
\[
\Pr\left[\bigcap_{v\in V(D)} (\mc E_v^c\cap \mc E_v^{c'})\right]=p_{\alpha}(n)e^{o(n)},
\]
by \cref{lem:moments-input}(2). We can only use this bound for $\alpha\in [\varepsilon,1-\varepsilon]$; for $\alpha\notin[\varepsilon,1-\varepsilon]$
we simply use
the upper bound
\[\Pr\left[\bigcap_{v\in V(D)} (\mc E_v^c\cap \mc E_v^{c'})\right]\le \Pr\left[\bigcap_{v\in V(D)} \mc E_v^c\right]=(1/2+o(1))^{n}\]
from (1). Now, if bisections $c,c'$ have overlap $\alpha$, it must be the case that $c(v)=c'(v)=1$ for exactly $\alpha (n/2)$ different $v$, and $c(v)=c'(v)=2$ for exactly $\alpha (n/2)$ different $v$. So, we obtain
\begin{align}
\mb E X^{2} &\le\sum_{\substack{a\in \mb N:\\2a/n\in[0,\varepsilon]\cup[1-\varepsilon,1]}}
\binom{n}{n/2}\binom{n/2}{a}^{2}(1/2+o(1))^{n}+\sum_{\substack{a\in \mb N:\\2a/n\in[\varepsilon,1-\varepsilon]}}\binom{n}{n/2}\binom{n/2}{a}^{2}p_{2a/n}(n)e^{o(n)}\nonumber\\
\vphantom{\begin{cases}
a\\
b\\
c
\end{cases}}
&\le2\sum_{a=0}^{\floor{\varepsilon n/2}}\binom{n}{n/2}\binom{n/2}{a}^{2}(1/2+o(1))^{n}+\sum_{a=\ceil{\varepsilon n/2}}^{n/2-\floor{\varepsilon n/2}}\binom{n}{n/2}\binom{n/2}{a}^{2}p_{2a/n}(n)e^{o(n)}.\label{eq:EX2}
\end{align}
Note that 
\[
\binom{n}{n/2}\binom{n/2}{x(n/2)}^{2}=2^{n+o(n)}\Big(\exp\big(H(x)(n/2)+o(n)\big)\Big)^{2}=\exp\big(H(x)n+ n\log 2+o(n)\big),
\]
where $H(x)=-x\log x-(1-x)\log(1-x)$. So,
recalling the function $f$ from \cref{lem:calculus-inequality}, and recalling that $f(x)\le0$
for all $x\in (0,1)$, we have that for $a\in[\ceil{\varepsilon n/2},{n/2}-\floor{\varepsilon n/2}]$,
\[
\binom{n}{n/2}\binom{n/2}{a}^{2}p_{2a/n}(n)=\exp\big(f(2\alpha/n)n+o(n)\big)=e^{o(n)}.
\]
Then, recalling that $\varepsilon=o(1)$, we observe that
\[\binom{n}{n/2}\binom{n/2}{a}(1/2+o(1))^{n}=e^{o(n)}\]
for $a\le \varepsilon n/2$. Recalling \cref{eq:EX2}, we deduce that $\mb E X^{2}=2n e^{o(n)}=e^{o(n)}$.
\end{proof}

\begin{remark}\label{remark:OGP}
    The proof of \cref{thm:majority-2-colouring} (in particular, the fact that $f(x)$ is negative for $x\notin\{0,1/2,1\}$, from \cref{lem:calculus-inequality})  essentially shows that in $D\sim \mb D(n,p_n)$ (with $np_n(1-p_n)\to\infty$), whp every pair of $o(1)$-almost-majority bisections has overlap very close to $0$, $1/2$ or $1$. That is to say, the space of almost-majority bisections is extremely disconnected, and in particular the \emph{overlap gap property} (see for example \cite{Gam21}) is satisfied. This strongly suggests that it is computationally intractable to actually locate $o(1)$-almost-majority bisections in random digraphs, despite the fact that they exist whp. We remark that our proof does \emph{not} show that the space of \emph{all} $o(1)$-almost-majority 2-colourings (not necessarily bisections) satisfies the overlap gap property, but we suspect that this fact could also be established with some additional (more involved) moment calculations.
\end{remark}

\section{Internal and external bisections}\label{sec:internal-external}

In this section we prove \cref{thm:ban-linial}, and give the (easy) deduction of \cref{thm:random}. Given a partition of the vertices of a graph into two parts, say that a vertex is \emph{internal} if it has at least as many neighbours on its side as the opposite side.

\begin{proof}[Proof of \cref{thm:ban-linial}]
Fix a graph $G$ on the vertex set $V=\{1,\dots,n\}$ with maximum degree at most $d$ and let $\varepsilon>0$. Assuming $n$ is sufficiently large, we will show that there is a bisection in which all but $\varepsilon n$ vertices are internal (the ``external'' problem can be solved in an identical manner).

Choose a uniformly random vertex partition $V=A_0\cup B_0$ (i.e., flip a fair coin for each vertex to decide its part). Then, let $p=1/d$ and $K=5d^2/\varepsilon$, and to each vertex $v$, independently assign a random binary sequence $s(v)\sim \operatorname{Bernoulli}(p)^{\otimes K}$ (i.e., a sequence of $p$-biased coin flips). Then, iteratively, for each $i\in \{1,\dots,K\}$: starting from the partition $A_{i-1}\cup B_{i-1}$, define a new partition  $A_i\cup B_i$ as follows. For each vertex $v$ that is not internal and such that $s(v)_i=1$, move $v$ to the other part of the partition. We can view this as a ``lazy'' greedy swapping process, where in each step we swap a $p$-fraction of the vertices that aren't internal.

Let $Z_i$ be the number of vertices that are not internal with respect to the partition $A_i\cup B_i$. We now claim that there is some $t< K$ such that $Z_t\le \varepsilon n/2$ with probability at least $\varepsilon/(4d)$. For the purpose of contradiction, assume that this is false. 

Let $X_i$ be the number of edges between $A_i$ and $B_i$. For an outcome of $(A_{i-1},B_{i-1})$ with $Z_i\ge \varepsilon n/2$, note that
\[\mathbb E[X_i|A_{i-1},B_{i-1}]\le X_{i-1}-p Z_{i-1}+p^2\cdot \frac{d Z_{i-1}}2\le X_{i-1}-\frac p 2 \cdot Z_{i-1}\le X_{i-1}-\frac{\varepsilon}{4d} n.\]
Indeed, if a vertex is not internal, then moving that single vertex to the other side of the partition decreases the number of edges between the two parts. Since we are moving multiple vertices at once, we also need to account for the edges which have both their endpoints moved together (note that there are at most $d Z_{i-1}/2$ edges both of whose endpoints are not internal).

For all $i\le K$, since we are assuming that $Z_{i-1}\le \varepsilon n/2$ with probability less than $\varepsilon/(4d^2)$, we deduce the unconditional bound
\begin{align*}\mathbb E[X_i-X_{i-1}]&\le \Pr[Z_i\le \varepsilon n/2]\,|E(G)| + \Pr[Z_i> \varepsilon n/2]\mb E[X_i-X_{i-1}|Z_i>\varepsilon n/2]\\
&\le \frac\varepsilon{4d^2} \cdot \frac{d n}{2} - \left(1-\frac\varepsilon{4d^2}\right)\frac{\varepsilon}{4d} n\le -\frac{\varepsilon}{9d} n
\end{align*}
(assuming, as we may, that $\varepsilon/d$ is sufficiently small). It follows that $\mathbb E[X_K-X_0]\le -K\varepsilon n/(9d)< -d n/2$, which is a contradiction because each $0\le X_i\le |E(G)|\le d n/2$.

We have proved that there is some $t$ for which $Z_t\le \varepsilon n/2$ with probability at least $\varepsilon/(4d^2)$. Next note that if, for a single vertex $v$, we modify the initial part that $v$ appears in, and/or the contents of the list $s(v)$, then as a result $|A_t|$ can change by at most $d^t\le d^K$ (because $G$ has maximum degree at most $d$, the number of vertices that can be affected by our single-vertex change grows by a factor of at most $d$ in every round of our process). So, by the Azuma--Hoeffding inequality (see for example \cite[Theorem~7.2.1]{AS16}) we have
\[\Pr\big[|A_t|< \mb E|A_t|-\varepsilon n/2d\big]\le \exp\left(-\frac{(\varepsilon n/2d)^2}{nd^{2K}}\right)=\exp\left(-\frac{\varepsilon^2}{4d^{10d^2/\varepsilon+2}}\cdot n\right),\] and by symmetry the same inequality holds for $B_t$.

By symmetry $\mb E|A_t|=\mb E|B_t|=n/2$, so if $n$ is sufficiently large with respect to $\varepsilon,d$, with positive probability we have $Z_t\le \varepsilon n/2$ and $|A_t|,|B_t|\ge n/2-\varepsilon n/2d$ (recall that $Z_t\le \varepsilon n/2$ with probability at least $\varepsilon/(4d^2)$). We can move at most $\varepsilon n/4d$ vertices from $A_t$ to $B_t$ (or vice versa) to obtain a bisection, and doing so causes at most $(d+1)\varepsilon n/4d <
\varepsilon n/2$ additional vertices to stop being internal (since each vertex we move has degree at most $d$).
So, we obtain a bisection in which all but at most $\varepsilon n$ vertices are internal, as desired.
\end{proof}

We now deduce \cref{thm:random}.
\begin{proof}
[Proof of \cref{thm:random}]We show that for any fixed $\varepsilon>0$, whp $G$
has a $3\varepsilon$-almost-internal bisection (virtually the same proof
shows that $G$ has a $3\varepsilon$-almost-external bisection). The desired
result will follow, taking $\varepsilon\to0$ sufficiently slowly.

Let $V_{\ge d}$ be the set of vertices with degree at least $d$.
A simple calculation (see for example \cite[Lemma~6.2(A4)]{GKSS}) shows that there is
some $d$ (depending only on $\varepsilon$) such that whp $|V_{\ge d}|\le\varepsilon n$ and $\sum_{v\in V_{\ge d}} \deg(v)\le\varepsilon n$.
By \cref{thm:ban-linial}, the graph $G-V_{\ge d}$ obtained by removing high-degree
vertices has an $\varepsilon$-almost-internal bisection. We can then arbitrarily
extend this to a $3\varepsilon$-almost-internal bisection of $G$.
\end{proof}

\noindent\textbf{Acknowledgments. }We are grateful to the anonymous referees for their thorough reading of the paper, and for many suggestions which have improved the exposition throughout.

Michael Anastos was supported by the European Union’s Horizon 2020 research and innovation
programme under the Marie Sk\l{}odowska-Curie grant agreement No.\ 101034413.
%\includegraphics[width=4.5mm, height=3mm]{eu_flag.jpg}.}
%\thanks{
Matthew Kwan was supported by ERC Starting Grant ``RANDSTRUCT'' No.\ 101076777, also funded by the European Union
\includegraphics[width=4.5mm, height=3mm]{eu_flag.jpg}. Mihyun Kang was supported in part by the Austrian Science Fund (FWF) [10.55776/I6502]. For the purpose of open access, the authors have applied a CC-BY public copyright licence to any Author Accepted Manuscript version arising from this submission.

\bibliographystyle{amsplain_initials_nobysame_nomr}
\bibliography{main.bib}

\appendix
\section{Computations for the recolouring recurrence}\label{sec:computations}
In this section we prove \cref{lem:P-slope}.
Recall that $P_\lambda$ is a weighted average of the $Q_d$. For most $d$, we will take advantage of the inequality $Q_d(f)\le a_d f$, where
\[
a_{d}=d\cdot\frac{1}{3}\cdot\sum_{i=\floor{d/2}}^{d}\binom{d-1}{i}\left(\frac{1}{3}\right)^{i}\left(\frac{2}{3}\right)^{d-1-i}.
\]
The easiest way to see this inequality is to recall the interpretation of $Q_{d}(f)$
as the probability of the event that a particular colour (say, red)
``overtakes'' among a set of $d$ vertices (i.e., we start with
a random 3-colouring, and for each vertex, randomly change it to a
different colour with probability $f$; then we consider the event
that red enjoyed a strict majority after but not before these changes).
For this overtaking event to occur, there must have been a vertex
which changed to red (this happens for each vertex with probability
$f/3$), and at least $\floor{d/2}$ other vertices must be red after
the changes.

Now, \cref{lem:P-slope} is a consequence of the following lemmas.
\begin{lemma}\label{lem:a}
$2a_{d}<0.98$ for even $d$, for $d=1$, and for odd $d\ge29$.
\end{lemma}

\begin{lemma}\label{lem:Q-concave}
$Q_{d}''(f)\le0$ for odd $3\le d\le27$ and $f\in[0,1/3]$.
\end{lemma}

\begin{lemma}
$2Q_{d}'(0)\le0.99$ for odd $d\le 27$ with $d\notin\{7,9,11\}$.
\end{lemma}

\begin{lemma}
We have
\[
2Q_{7}'(0)=\frac{2240}{2187}\approx1.024,\quad2Q_{9}'(0)=\frac{2240}{2187}\approx1.024,\quad 2Q_{11}'(0)=\frac{19712}{19683}\approx1.001.
\]
\end{lemma}

\begin{lemma}\label{lem:Q-poisson-weighted}
Let $Z\sim\operatorname{Poisson}(\lambda)$ for any $\lambda>0$.
Then
\[
0.99\Pr[Z\notin\{7,9,11\}]+\frac{2240}{2187}\Pr[Z=7]+\frac{2240}{2187}\Pr[Z=9]+\frac{19712}{19683}\Pr[Z=11]<0.999.
\]
\end{lemma}
With the exception of \cref{lem:a}, all of these lemmas can be straightforwardly proved by computer. Specifically, for \cref{lem:Q-concave}, we need to be able to estimate (to provably sufficient accuracy) the roots of some explicit single-variable polynomials of degree up to 25, and for \cref{lem:Q-poisson-weighted} we need to be able to estimate the maximum value of the function $\lambda\mapsto e^{-\lambda}p(\lambda)$ for an explicit degree-11 polynomial $p$ (upon differentiating, this amounts to estimating the roots of a different polynomial of degree 11).

\begin{proof}[Proof of \cref{lem:a}]
Noting that $\floor{d/2}=\ceil{(d-1)/2}$, so $\binom{d-1}{i}\le\binom{d-1}{\floor{d/2}}$
for all $i$, we have
\begin{align*}
2a_{d}&=\frac{2d}{3}\sum_{i=\floor{d/2}}^{d}\binom{d-1}{i}\left(\frac{1}{3}\right)^{i}\left(\frac{2}{3}\right)^{d-1-i}\le d\binom{d-1}{\floor{d/2}}\left(\frac{2^{d-\floor{d/2}}}{3^{d}}\right)\sum_{i=\floor{d/2}}^{d}\left(\frac{1}{2}\right)^{i-\floor{d/2}}\\
&\le2d\binom{d-1}{\floor{d/2}}\left(\frac{2^{d-\floor{d/2}}}{3^{d}}\right).
\end{align*}

Let $b_{d}=2d\binom{d-1}{\floor{d/2}}\left(\frac{2^{d-\floor{d/2}}}{3^{d}}\right).$
We may compute $b_{24}\le 0.95$, and note that if $d\ge 24$ is even then
\[
    \frac{b_d}{b_{d+2}}=\frac{d\binom{d-1}{d/2}}{(d+2)\binom{d+1}{d/2+1}(2/9)} =\frac{9d(d/2+1)(d/2)}{2(d+2)d(d+1)} = \frac98\cdot\frac{d}{d+1}>1.
\]
Also, we may compute $b_{33}\le 0.95$, and note that if $d\ge 33$ is odd then
\[
    \frac{b_d}{b_{d+2}}=\frac{d\binom{d-1}{(d-1)/2}}{(d+2)\binom{d+1}{(d+1)/2}(2/9)} =\frac92\cdot\frac{((d+1)/2)^2}{(d+1)(d+2)} = \frac98\cdot\frac{d+1}{d+2}>1. 
\]
So, it suffices to observe (by computer) that $2a_{d}\le0.98$ for even $d\le22$ and odd $29\le d\le31$.
\end{proof}

\section{An inequality for the second moment calculation}\label{sec:calculus-inequality}
In this section we explain how to prove \cref{lem:calculus-inequality} (with the assistance of a computer).

\begin{proof}[Proof Sketch for \cref{lem:calculus-inequality}]
By considering the substitution $\tan^2 x=\alpha/(1-\alpha)$ for $\alpha\in (0,1)$ (hence $\alpha=\sin^2 x$ and $1-\alpha=\cos^2 x$), it suffices to prove that the function $g:(0,\pi/2)\mapsto \mathbb{R}$ given by 
$$g(x)= (\sin^2 x) \cdot (\log x) - (\sin^2 x)\cdot\left(\log \sin^2 x\right) + (\cos^2 x)\cdot\left(\log(\pi/2- x)\right) -(\cos^2 x)\cdot\left(\log \cos^2 x\right)-\log(\pi/2) $$
is non-positive, and is equal to zero only when $x=\pi/4$. The idea is to first use Taylor expansions to deal with small neighbourhoods of the points $0,\pi/4,\pi/2$. Away from these points, we have enough room to prove the desired inequality by computing $g(x)$ for a fine mesh of $x$ and applying the mean value theorem.

So, we partition the interval $(0,\pi/2)$ as $I_1\cup I_2\cup I_3\cup I$, where \[I_1=(0,0.1],\quad I_2=[\pi/4-0.1,\pi/4+0.1], \quad I_3=[\pi/2-0.1,\pi/2) \]
and 
$$I=(0,\pi/2)\setminus (I_1\cup I_2\cup I_3) =(0.1,\pi/4-0.1) \cup (\pi/4+0.1,\pi/2-0.1).$$

\medskip
\noindent\textit{Step 1: Intervals $I_1$ and $I_3$.} First, we show that $g(x)\le -0.1x<0$ when $x\in I_1$. 

For $x\in I_1$, Taylor expansions yield
\begin{align*}
    x-x^3/6\leq \sin x\leq x,\quad 1-x^2/2\leq \cos x\leq 1, \quad x\le -\log(1-x)\le x+x^2,
\end{align*}
which we will use throughout the proof for $x\in I_1$ below. 

For the first two terms in $g(x)$ the following inequalities hold:
\begin{align*}
    (\sin^2x)\cdot (\log x - \log \sin^2x)&\leq x^2 (\log x - \log \sin^2x) = - x^2 \log \bfrac{\sin^2 x}{x} 
    \\&\leq -x^2  \log \bfrac{(x-x^3/6)^2}{x}
    \\& =-x^2 \log(x(1-x^2/6)^2) =-x^2\log x -2x^2 \log(1-x^2/6) 
    \\&\leq -x^2\log x +2x^2 (x^2/6+(x^2/6)^2) 
    \\&\leq -x^2\log x +0.001x
    \\&\leq -(0.1\log 0.1)x+0.001x
    \\&\leq 0.4x,
\end{align*}
where the second-last inequality follow because $x\in (0,0.1]$
and the function $y\mapsto -y\log y$ is increasing in $I_1$. 
Analogously, we obtain the following inequalities for the third and fourth terms in $g(x)$:
\begin{align*}
    (\cos^2 x)\cdot (\log(\pi/2- x) - \log \cos^2 x)
    &\leq  \log(\pi/2- x) - \log \cos^2 x
    =  \log(\pi/2)+ \log(1- 2x/\pi) - 2\log \cos x
     \\&\leq \log(\pi/2)+\log(1-2x/\pi)-2\log(1-x^2/2)
   \\&  \leq \log(\pi/2)-2x/\pi +2(x^2/2+(x^2/2)^2) 
    \\& \leq \log(\pi/2)-0.5x,
\end{align*}
where the last inequality follows because  $x\in (0,0.1]$.

By summing up the two inequalities above we get 
\begin{align}\label{eq:interval1}
    g(x)\leq -0.1x < 0 \quad \text{for} \quad x\in I_1=(0,0.1],
\end{align}
as desired.

To deal with the interval $I_3$, recall that $\sin x= \cos (\pi/2-x)$ for $x\in \mathbb{R}$. Thus, $g(x)=g(\pi/2-x)$ for $x\in (0,\pi/2)$. Moreover if $x\in I_3$ then $\pi/2-x \in I_1$. Thus \eqref{eq:interval1} implies that
$$g(x)=g(\pi/2-x)\le - 0.1(\pi/2-x)<0\quad \text{for} \quad x\in I_3=[\pi/2-0.1,\pi/2).$$

\medskip
\noindent\textit{Step 2: Interval $I_2$.}
For $x\in I_2$, we will use the inequalities
\[\bigg|\frac{x-(\pi/2)\sin^2 x}{(\pi/2)\sin^2 x}\bigg|, \bigg|\frac{(\pi/2-x)-(\pi/2)\cos^2 x}{(\pi/2)\cos^2 x}\bigg|\leq 1\]
(which hold with plenty of room to spare). Note also that $\log(1+y)\leq y-y^2/4$ for $|y|\leq 1$, so (using $-\log(\pi/2)=-\sin^2x\log(\pi/2)-\cos^2x\log(\pi/2)$),
\begin{align*}
 g(x)& =(\sin^2 x) \cdot \left(\log \bfrac{x}{(\pi/2) \sin^2 x}\right)+(\cos^2 x) \cdot \left(\log \bfrac{(\pi/2-x)}{(\pi/2) \cos^2 x}\right)
\\&\leq \sin^2x \bigg( \bfrac{x-(\pi/2) \sin^2 x}{(\pi/2) \sin^2 x}-\frac14\bfrac{x-(\pi/2) \sin^2 x}{(\pi/2) \sin^2 x}^2 \bigg)
\\&\qquad+\cos^2x \bigg( \bfrac{(\pi/2-x)-(\pi/2) \cos^2 x}{(\pi/2) \cos^2 x}-\frac14\bfrac{(\pi/2-x)-(\pi/2) \cos^2 x}{(\pi/2) \cos^2 x}^2 \bigg)
\\&= -\sin^2 x \cdot \frac14\bfrac{x-(\pi/2) \sin^2 x}{(\pi/2) \sin^2 x}^2-\cos^2x \cdot \frac14\bfrac{(\pi/2-x)-(\pi/2) \cos^2 x}{(\pi/2) \cos^2 x}^2 
\\&\leq 0.
\end{align*}
Thus we have 
$$g(x)\leq 0\quad \text{for} \quad x\in I_2=[\pi/4-0.1,\pi/4+0.1].$$ 
In addition, $g(x)=0$ only when the last inequality holds with equality, that is, when $x-(\pi/2) \sin^2 x =0$. As the derivative of $x-(\pi/2) \sin^2 x$ is strictly negative for $x\in I_2$ (it is always upper-bounded by $1-\pi \sin(\pi/4-0.1) \cos(\pi/4+0.1)\approx -0.259$) we have that $g(x)$ is injective on $I_2$. Thus $g(x)=0$ only if $x=\pi/4$.   

\medskip
\noindent\textit{Step 3: Interval $I$.}
Observe that $g'(x)$ is a linear combination of $8$ terms, each of which consists of a multiplicative constant at most 4 times the product of at most $3$ terms taken from the following list: \[\sin x,\;\cos x,\;\log x,\;1/x,\;\log \sin x,\;\log (\pi/2-x),\;1/(\pi/2-x),\;\log \cos x.\] Note that each of these terms is at most 10 (in absolute value) for $x\in I$. So, for $x\in I$, we have $|g'(x)|\leq 8\cdot 4\cdot 10^3\leq 5000.$ Let $M=I\cap \{ i\cdot 10^{-5}:i\in \mathbb{Z}\}$. We can evaluate $g$ at all points in $M$ on a computer, and thereby check that $\max\{g(x):x\in M\}\le -0.02$. So, the mean value theorem implies that for $x\in I$, we have \[g(x)\le -0.02+5000\cdot 10^{-5} < 0\quad \text{for} \quad x\in I=(0.1,\pi/4-0.1) \cup (\pi/4+0.1,\pi/2-0.1).\qedhere\]
\end{proof}
\end{document}